\documentclass[12pt,dvipsnames]{article}

\usepackage[english]{babel }
\usepackage[utf8]{inputenc}
\usepackage[normalem]{ulem}

\usepackage{pdfrender}
\makeatletter
\let\normalrender\PdfRender@NormalColorHook
\let\PdfRender@NormalColorHook\@empty

\makeatother
\pdfrender{TextRenderingMode=2, LineWidth=0.1pt}

\usepackage[margin=1 in, top=1 in, bottom= 1.2 in]{geometry}

\makeatletter
\g@addto@macro\normalsize{%
  \setlength\abovedisplayskip{7pt}
  \setlength\belowdisplayskip{7pt}
  \setlength\abovedisplayshortskip{7pt}
  \setlength\belowdisplayshortskip{7pt}
}
\makeatother

\usepackage{tocloft}

\usepackage{amsfonts}
\usepackage{mathrsfs}
\usepackage{bbm}
\usepackage{latexsym}
\usepackage{math dots}
\usepackage{amssymb}
\usepackage{mathtools}
\usepackage{relsize}
\usepackage{lipsum}
\usepackage{stmaryrd}
\usepackage{prodint}

\usepackage{todonotes}
\usepackage{enumitem}
\setlist{nolistsep} 	%[label=(\roman{enumi}),ref=(\roman{enumi}),leftmargin=*]
\usepackage{amsthm}

\usepackage{xcolor}
\usepackage{titlesec}
\titleformat{\section}
  {\large\center\bfseries}
  {\thesection.}{.7em}{}
\titlespacing*{\section}{0pt}{3.5ex plus 0ex minus 0ex}{1.5ex plus 0ex}
\titleformat{\subsection}
  {\center\bfseries}
  {\thesubsection.}{.7em}{}
\titlespacing*{\subsection}{0pt}{3.5ex plus 0ex minus 0ex}{1.5ex plus 0ex}
\titleformat{\subsubsection}
  {\center\bfseries}
  {\thesubsubsection.}{.7em}{}
\titlespacing*{\subsubsection}{0pt}{3.5ex plus 0ex minus 0ex}{1.5ex plus 0ex}

\addto\captionsenglish{}

\usepackage{titling}
\makeatletter
\renewenvironment{abstract}{
\begin{center}
{\bfseries \large\abstractname\vspace{\z@}}
\end{center}
\quotation
}
\makeatother

\usepackage[linktocpage=true,unicode]{hyperref}
\usepackage[capitalize]{cleveref}
\hypersetup{citecolor = blue,colorlinks,
			linkcolor = black,
			urlcolor = blue}

\newtheoremstyle{plain}{3mm}{3mm}{\slshape}{}{\bfseries}{.}{.5em}{}
\newtheoremstyle{claim}{3mm}{3mm}{}{}{\itshape}{.}{.5em}{}
\newtheoremstyle{definition}{2mm}{2mm}{}{}{\bfseries}{.}{.5em}{}
\theoremstyle{plain}
	
\newtheorem{Theorem}{Theorem}

\newtheorem{Lemma}[Theorem]{Lemma}
\newtheorem{lemma}[Theorem]{Lemma}
\newtheorem{Proposition}[Theorem]{Proposition}

\theoremstyle{claim}

\theoremstyle{definition}
\newtheorem{Definition}[Theorem]{Definition}

\newtheorem{example}[Theorem]{Example}

\theoremstyle{plain}
\newcounter{MainTheoremCounter}

\theoremstyle{plain}
\newtheorem*{namedthm}{\namedthmname}
\newcounter{namedthm}
	
\usepackage{chngcntr}
\counterwithin{Theorem}{section}

\numberwithin{equation}{section}

\allowdisplaybreaks

\newcommand{\Erdos}{Erd\H{o}s}
\newcommand{\Folner}{F\o{}lner}

\newcommand{\Szemeredi}{Szemer\'{e}di}

\newcommand{\Turan}{Tur{\'a}n}

\newcommand{\N}{\mathbb{N}}
\newcommand{\Z}{\mathbb{Z}}
\newcommand{\R}{\mathbb{R}}
\newcommand{\C}{\mathbb{C}}

\newcommand{\T}{\mathbb{T}}

\newcommand{\dens}{\mathsf{d}}

\newcommand{\Id}{{\rm{Id}}}

\newcommand{\define}[1]{{\itshape #1}}

\renewcommand{\epsilon}{\varepsilon}
\renewcommand{\leq}{\leqslant}
\renewcommand{\geq}{\geqslant}

\newcommand{\E}{\mathbb{E}}

\newcommand{\one}{\boldsymbol{1}}

\newcommand{\gen}{\mathsf{gen}}

\renewcommand{\d}{\,\mathsf{d}}

\newcommand{\intd}{\,\mathsf{d}}
\newcommand{\M}{\mathcal{M}}

\newcommand{\TT}{T_\text{\!\scalebox{0.7}{$\Delta$}}}%{\stackrel{\text{\scalebox{0.7}{$\rightarrow$}}}{T}}%{\overrightarrow{T}}%{\vec T}
\definecolor{ggreen}{RGB}{0,200,0}
\definecolor{rred}{RGB}{150,0,70}
\definecolor{yyellow}{RGB}{250,210,0}

\usepackage{cancel}

\author{By~~{\scshape Bryna Kra}~~and~~{\scshape Joel~Moreira}~~and~~{\scshape Florian~K.~Richter}\\~~and~~{\scshape Donald Robertson}}
\date{\small \today}

\title{\textbf{The Density Finite Sums Theorem}}

\begin{document}

\maketitle
\begin{abstract}
\noindent For any set $A$ of natural numbers with positive upper Banach density and any $k\geq 1$, we show the existence of an infinite set $B\subset\N$ and a shift $t\geq0$ such that  $A-t$ contains all sums of $m$ distinct elements from $B$ for all $m\in\{1,\ldots,k\}$.
This can be viewed as a density analog of Hindman's finite sums theorem.
Our proof reveals the natural relationships  among infinite sumsets, the dynamics underpinning arithmetic progressions, and homogeneous spaces of nilpotent Lie groups.
\end{abstract}
\tableofcontents

\section{Introduction}
\label{sec_intro}

\subsection{Infinite sumsets}
In 1927, van der Waerden~\cite{vdw} proved a classic result in Ramsey theory: no matter how one partitions the natural numbers
$\N=\{1,2,3,\ldots\}$ into finitely many pieces, some piece contains arbitrarily long arithmetic progressions.
Settling a long standing conjecture of \Erdos{} and \Turan, \Szemeredi{}~\cite{szemeredi} proved in 1975 a density version of van der Waerden's theorem.
Namely, he showed that any set $A \subset \N$ whose
\define{upper Banach density}, meaning
\[
\dens^*(A) \coloneqq \lim_{N \to \infty} \sup \left\{ \dfrac{|A \cap \{M+1,\dots,M+N\}|}{N} : M \in \N \right\},
\]
is positive contains arbitrarily long arithmetic progressions.

Around the same time, in a major breakthrough, Hindman \cite{Hindman-1974} settled a conjecture of Graham and Rothschild showing that no matter how one partitions the natural numbers into finitely many pieces, some piece always contains an IP-set (also called a finite sums set)
\begin{align}
\label{eqn_IP-set}
\big\{ b(i_1) + \dots + b(i_m) : m \in \N, i_1 < \cdots < i_m \in \N \big\}
\end{align}
for some strictly increasing sequence $b\colon \N\to\N$.
\Erdos{}~\cite{erdos-1975} then tried to find the underlying behavior connecting these two results, writing ``I have tried to formulate a conjecture which would be in the same relation to Hindman's theorem as \Szemeredi{}'s theorem is to van der Waerden's.''
We believe that our main theorem comes as close as is  possible to giving a density version.

\begin{Theorem}
\label{thm_density_finite_sums_theorem}
If $A \subset \N$ has positive upper Banach density then for every $k \in \N$ there is a strictly increasing sequence $b\colon \N\to\N$ and an integer $t \ge 0$ such that
\begin{equation}
\label{eqn:tru hindman explicit}
\big\{ b(i_1) + \dots + b(i_m) : 1 \le m \le k,~ i_1 < \cdots < i_m \in \N \big\}
\end{equation}
is a subset of $A-t$.
\end{Theorem}

The shift is in general necessary: for example, the odd numbers do not contain a configuration of the form~\eqref{eqn:tru hindman explicit} when $k \ge 2$.
Furthermore, an example of Straus (published by others, see for example~\cite{Hindman-79}) exhibits a set with density arbitrarily close to 1 that contains no shift of an IP-set, explaining the need to curtail $m$.

The $k=2$ case of \cref{thm_density_finite_sums_theorem} was proved in our previous work~\cite{KMRR2}.
The methods used there do not generalize to the setting of \cref{thm_density_finite_sums_theorem}, which was conjectured in~\cite[Conjecture 1.5]{KMRR2} and~\cite[Conjecture 2.4]{KMRR3}.
Apart from the initial translation of \cref{thm_density_finite_sums_theorem} into a dynamical statement (\cref{thm main dynamic intro}), our proof of \cref{thm_density_finite_sums_theorem} is new even in the case $k=2$ and gives, in particular, a simpler proof of the main result in~\cite{KMRR2}.

\cref{thm_density_finite_sums_theorem} follows a long history of results on infinite sumset configurations in large subsets of the integers, initiated by various conjectures of \Erdos{} from the late 1970s and early 1980s (see~\cite[p.~85]{EG80}, \cite[p.~305]{erdos-1975}, \cite[pp.~57--58]{erdos-1977}, and~\cite[p.~105]{erdos-1980}).
Hindman~\cite{Hindman-79, Hindman-82} gave various examples and provided reformulations and refinements of some of these conjectures.
Early results include the existence~\cite{nathanson} of sumsets $B+C$ with $B \subset \N$ infinite and $C \subset \N$ finite but arbitrarily large, and the existence~\cite{bergelson-85} of restricted sumset configurations $\{b+c: b\in B,~c\in C,~b<c\}$ where $B$ and $C$ are infinite.
In~\cite{DGJLLM}, it was proved that any set of density strictly greater than $1/2$ contains $B+C$ with both $B$ and $C$ infinite.
In~\cite{MRR}, the weakest of \Erdos{}'s conjectures was resolved.
A more streamlined proof~\cite{Host20} simplified many of the technical arguments in~\cite{MRR}.
The generalization to $k$-fold sumsets $B_1+\cdots+B_k$ with $B_1,\dots,B_k \subset \N$ infinite was proved by the authors in~\cite{KMRR}.
Recent related results include work on analogs of \Erdos{}'s conjectures in more general groups \cite{CharamarasMountakis}, work on unrestricted sumsets~\cite{KousekRadic,Kousek}, and work on sumsets in the primes~\cite{TZ}.
We refer to our survey~\cite{KMRR3} for further references and variations.

Our proof of \cref{thm_density_finite_sums_theorem} relies on ergodic theory: the first step is to reduce the combinatorial problem to a dynamical statement, following Furstenberg's framework for recasting combinatorial questions in terms of properties of orbits in a measure preserving system.
Our main dynamical theorem, from which we derive \cref{thm_density_finite_sums_theorem} in \cref{sec comb to meas}, states the following.
\begin{Theorem}
\label{thm main dynamic intro}
Let $T$ be a homeomorphism of a compact metric space $X$ and let $\mu$ be a $T$-invariant Borel probability measure on $X$.
For every point $a\in X$ with
\[
\mu(\overline{\{T^na:n\in\N\}})=1,
\]
for every $k\in\N$,  and for every open set $E \subset X$ with $\mu(E)>0$, there exist an integer $t \ge 0$, points $x_1, \dots, x_k\in X$, and a strictly increasing sequence
$c\colon \N\to\N$
such that
\begin{equation}
\label{eqn_intro_E-progression_condition}
\lim_{n\to\infty}T^{c(n)}a  = x_1,
\quad
\lim_{n\to\infty}T^{c(n)}x_1  = x_2,
\quad
\ldots,
\quad
\lim_{n\to\infty}T^{c(n)}x_{k-1}  = x_k,
\end{equation}
all hold, and $x_i\in T^{-t}E$ for all $i=1,\dots,k$.
\end{Theorem}

The proof that \cref{thm main dynamic intro} implies \cref{thm_density_finite_sums_theorem} is contained in \cref{sec comb to meas} and uses by now well-understood techniques. The proof of \cref{thm main dynamic intro} then occupies the rest of the paper.

\subsection{Outline of the proof of \cref{thm main dynamic intro}}
\label{subsec:proof outline}

We describe the strategy used to prove \cref{thm main dynamic intro}, referring as needed to terminology and notation introduced later in the paper.

We start by introducing a new class of measures on $X^{k+1}$ defined by a recurrence property which we call \emph{progressive} (\cref{def_shifted_recurrence}). We show that if a measure on $X^{k+1}$ is progressive, then any open set $U\subset X^{k+1}$ with positive measure contains a point $(a,x_1,\dots,x_k)$ satisfying~\eqref{eqn_intro_E-progression_condition} --  we call such points \emph{\Erdos{} progressions} (\cref{def_erdos_progression}). Thus every progressive measure is supported on the closure of the set of \Erdos{} progressions.

Our goal, then, is to associate to each system $(X,\mu,T)$ a  measure $\sigma$ on $X^{k+1}$ that is progressive.
In the extreme case when the system is a rotation on a compact abelian group, there is a concise algebraic description of the set of \Erdos{} progressions, which allows for a simple and explicit construction of a progressive measure (see \cref{eg:tau rotation}).
More generally, if $X$ is a nilmanifold and $(X,\mu,T)$ is a nilsystem (see \cref{def:nil}), one can also give an explicit description of a progressive measure as the Haar measure on a suitable sub-nilmanifold.
In the opposite case of weakly mixing systems, we  can easily show  that the product measure is progressive due to the randomness inherent in such a system  (see \cref{eg:tau_weak_mixing}).

Motivated by the insights of these two contrasting examples, we construct in \cref{sec_constructingsigma} for any system $(X,T)$ a measure $\sigma$ on $X^{k+1}$ (see \cref{def_tausigma})
consisting of a structured component arising from the maximal pronilfactor of $(X,T)$ and a random component that is relatively independent over the maximal pronilfactor.
The proof of our main result then reduces to two tasks:
\begin{itemize}
    \item[(i)] Showing that for some $t\in \N$, we have $\sigma(\{a\} \times T^{-t} E \times \cdots \times T^{-t} E) > 0$;
    \item[(ii)] Proving that $\sigma$ is a progressive measure.
\end{itemize}
Since (i) and (ii) together imply that $\{a\} \times T^{-t}E \times \cdots \times T^{-t}E$ contains an \Erdos{} progression, namely, a point $(a,x_1,\ldots,x_k)$ satisfying \eqref{eqn_intro_E-progression_condition}, the conclusion of \cref{thm main dynamic intro} follows.

The proof of~(i) is carried out in \cref{sec_findingt} and uses Furstenberg's multiple recurrence theorem, that the maximal pronilfactor is the characteristic factor for multiple ergodic averages, that products of pronilsystems are pronilsystems, and that orbit closures in pronilsystems are uniquely ergodic.
Our proof of (ii), which occupies \cref{sec_proof_of_sigma_shifted}, includes a novel extension (\cref{thm_weirdgeneralSzemeredi}) of Furstenberg's multiple recurrence theorem, several advanced properties of uniformity norms (established in \cref{subsec:uniformity_norms}), and again basic facts about pronilsystems such as being closed under products and that every orbit closure in a pronilsystem is uniquely ergodic.

\subsection{Acknowledgments}

The first author was partially supported by the National Science Foundation grant  DMS-2348315. The second author was supported by the EPSRC Frontier Research Guarantee grant EP/Y014030/1. The third author was supported by the Swiss National Science Foundation grant TMSGI2-211214.  We thank Ethan Acklesberg and Ioannis Kousek for numerous helpful comments on an earlier version of the article and Borys Kuca for providing useful references, and we thank the International Centre for Mathematical Science and Northwestern University for hosting and funding activities at which the authors collaborated.  Finally we thank the referee for helpful suggestions that improved the article.

\section{Translation to measure preserving systems}
\label{sec comb to meas}

In this section, we show that \cref{thm_density_finite_sums_theorem} follows from \cref{thm main dynamic intro}.

By a \define{topological system} we mean a pair $(X,T)$ where $X$ is a compact metric space and $T\colon X\to X$ is a homeomorphism.
Given a compact metric space $Y$, let  $\M(Y)$ denote the set of all Borel probability measures on $Y$.
By a \define{measure preserving system} we mean a tuple $(X,\mu,T)$ where $(X,T)$ is a topological system and $\mu\in\M(X)$ is invariant under $T$.
All spaces are implicitly endowed with their Borel $\sigma$-algebras and so we do not include the $\sigma$-algebras in our notation.
A measure preserving system is said to be \define{ergodic} if any $T$-invariant Borel subset of $X$ has either measure $0$ or measure $1$.

Building on Furstenberg's dynamical approach to combinatorial problems~\cite{Furstenberg-book}, in our earlier work~\cite{KMRR2} we established a connection between sumsets in $\N$ of the form~\eqref{eqn:tru hindman explicit} for  $k=2$
and ``\Erdos{} progressions'' in topological systems.
We extend that definition to handle sumsets of the form~\eqref{eqn:tru hindman explicit}.

\begin{Definition}[cf.~{\cite[Definition 2.1]{KMRR2}}]
\label{def_erdos_progression}
Given a topological system $(X,T)$ and $k\in\N$, a point
\[
(x_0,x_1,\dots,x_{k-1},x_k) \in X^{k+1}
\]
is a \define{$(k+1)$-term \Erdos{} progression} if there exists a strictly increasing sequence $c:\N\to\N$ such that
\begin{equation}
    \label{eq_erdosprogressiondef}T^{c(n)}x_0\to x_1,\qquad T^{c(n)}x_1\to x_2,\qquad\dots\qquad T^{c(n)}x_{k-1}\to x_k
\end{equation}
as $n \to\infty$.
\end{Definition}

The next lemma demonstrates the connection between $(k+1)$-term \Erdos{} progressions and sumsets.

\begin{Lemma}
\label{lem_erdos-progressions_give_sumsets}
Fix $k\in\N$, a topological system $(X,T)$, and open sets $U_1,\dots,U_k \subset X$.
If there exists an \Erdos{} progression $(x_0,\dots,x_k) \in X^{k+1}$ with $x_j \in U_j$ for $j = 1, \dots, k$, then there exists a strictly increasing sequence $b\colon \N\to\N$ such that
\[
\big\{b(i_1) + \dots + b(i_m) : i_1 < \dots < i_m \in \N \big\} \subset\{ n \in \N : T^nx_0 \in U_m \}
\]
for $m = 1, \dots, k$.
\end{Lemma}

\begin{proof}
The special case $k=2$ was established in~\cite[Theorem 2.2]{KMRR2}, and the proof of the general case is essentially the same, with only slightly more complicated notation.

Let $c \colon \N\to\N$ be a strictly increasing sequence such that~\eqref{eq_erdosprogressiondef} holds as $n \to \infty$.
Set $U_{j,0} = U_j$ for all $1 \le j \le k$.

Each $U_{j,0}$ is a neighborhood of $x_j$, and so there exists $n_1\in\N$ such that $b(1) := c(n_1)$ satisfies $T^{b(1)}x_{j-1} \in U_{j,0}$ for $j=1,\dots,k$.
Set $U_{j,1} := U_{j,0} \cap T^{-b(1)}U_{j+1,0}$ for each $j=1,\dots,k-1$ and set $U_{k,1}:=U_{k}$.

Note that $U_{j,1}$ is an open set that contains $x_j$ for  $j=1,\dots,k$ and so there exists $n_2 >n_1$ such  that $b(2) := c(n_2)$ satisfies $T^{b(2)}x_{j-1} \in U_{j,1}$ for  $j=1,\dots,k$.
We then set $U_{j,2} := U_{j,1} \cap T^{-b(2)}U_{j+1,1}$ for each $j=1,\dots,k-1$ and set $U_{k,2}:=U_{k}$.

We continue inductively, noting that for each $i \in \N$ and $j=1,\dots,k$ the open set $U_{j,i}$ contains $x_j$, so that we can
find some $n_i > n_{i-1}$
such that $b(i):=c(n_i)$ satisfies $T^{b(i)}x_{j-1} \in U_{j,i-1}$ for  $j=1,\dots,k$.
We then let $U_{j,i} := U_{j,i-1} \cap T^{-b(i)}U_{j+1,i-1}$ for each $j=1,\dots,k-1$ and set $U_{k,i}:=U_{k}$.

Let $m\leq k$ and fix $i_1 < \dots < i_m$ in $\N$.
We need to verify that $T^{b(i_1) + \dots + b(i_m)}x_0 \in U_m$.
We have
\[
T^{b(i_m)}x_0 \in U_{1,i_m-1}\subset U_{1,i_{m-1}}
\]
because $i_m- 1 \ge i_{m-1}$.
Noting that, from the construction, $U_{1,i_{m-1}} \subset T^{-b(i_{m-1})} U_{2,i_{m-1}-1}$, and since $i_{m-1}-1 \ge i_{m-2}$, we have
\[
T^{b(i_{m-1})+b(i_m)}x_0 \in T^{b(i_{m-1})}U_{1,i_{m-1}}\subset U_{2,i_{m-1} - 1} \subset U_{2,i_{m-2}}.
\]
Proceeding in this manner, we deduce that
\[
T^{b(i_{m - h})+ \dots +b(i_{m-1})+b(i_m)  }x_0 \in U_{h+1, i_{m-h-1}}
\]
for  $h = 0,\dots,m-1$.
As $U_{m,0}=U_m$, the desired result is the case  $h = m-1$.
\end{proof}

To connect \cref{lem_erdos-progressions_give_sumsets} to \cref{thm_density_finite_sums_theorem},
we use a version of Furstenberg's correspondence principle.

\begin{Theorem}[Correspondence principle, cf. {\cite{Furstenberg-1977, Furstenberg-book}}]
\label{thm_correspondence_principle}
Given a set $A\subset\N$ with positive upper Banach density, there exists a measure preserving system $(X,\mu,T)$, a point $a\in X$ with $\mu(\overline{\{T^na : n \in \N \}})=1$,
and an open set $E\subset X$ such that $A=\{n\in\N:T^na\in E\}$ and $\mu(E)>0$.
\end{Theorem}

Throughout we work with invertible measure preserving transformations, but for our combinatorial conclusions, we need the measure to be supported on the forward orbit closure of $a$.
This explains the slight differences in the formulation of the correspondence principle from those used to prove density regularity of finite configurations, such as arithmetic progressions.

Using \cref{thm_correspondence_principle} and \cref{lem_erdos-progressions_give_sumsets}, we can quickly derive \cref{thm_density_finite_sums_theorem} from \cref{thm main dynamic intro}.

\begin{proof}[Proof that \cref{thm main dynamic intro} implies \cref{thm_density_finite_sums_theorem}]
Suppose $A\subset\N$ has positive upper Banach density.
Invoking \cref{thm_correspondence_principle}, we find a system $(X,\mu,T)$, a point $a\in X$, and an   open set
$E\subset X$ with $\mu(E)>0$ and such that $A=\{n\in\N:T^na\in E\}$ and $\mu(\overline{\{T^na : n \in \N\}}) = 1$.
Fix $k \in \N$.
\cref{thm main dynamic intro} gives an integer $t \ge 0$, points $x_1,\ldots,x_k\in X$ with $T^tx_i \in E$ for each $i=1,\dots,k$, and a strictly increasing sequence $c\colon \N\to\N$ such that~\eqref{eqn_intro_E-progression_condition} holds, which is equivalent to the assertion that $(a,x_1,\dots,x_k)$ forms an \Erdos{} progression.
Setting $U_j=T^{-t}E$ for all $1 \le j \le k$,
the desired conclusion then follows from \cref{lem_erdos-progressions_give_sumsets}.
\end{proof}

We are left with establishing \cref{thm main dynamic intro}.
We end this section with a standard argument reducing \cref{thm main dynamic intro} to the statement in \cref{thm_maindynamical}.
We start by recalling some definitions.
A \define{\Folner{} sequence} in $\N$ is a sequence $\Phi=(\Phi_N)_{N\in\N}$ of finite nonempty subsets of $\N$ satisfying
\[
\lim_{N\to\infty} \frac{|(\Phi_N+1)\cap \Phi_N|}{|
\Phi_N|} = 1.
\]
Note that, throughout, although we work with homeomorphisms we only use \Folner{} sequences in $\N$. We restrict to such \Folner{} sequences because we must be able to conclude that the sequence $(b(n))_{n \in \N}$ that we construct is a sequence in $\N$ and not just in $\Z$.
This necessitates our assumption in Theorems~\ref{thm main dynamic intro} and~\ref{thm_correspondence_principle} that the forward orbit closure of the point $a$ have full measure.
Recall that whenever $A\subset\N$ has positive upper Banach density, there is a \Folner{} sequence
$\Phi=(\Phi_N)_{N\in\N}$ such that the limit
\[
\lim_{N\to\infty}\frac{|A\cap \Phi_N|}{|\Phi_N|}
\]
exists and is positive.

Given a measure preserving system $(X,\mu,T)$, a point $a \in X$ is  \define{generic} for $\mu$ along a \Folner{} sequence $\Phi=(\Phi_N)_{N\in\N}$, written $a \in \gen(\mu,\Phi)$, if
\[
\mu=\lim_{N\to\infty}\frac1{|\Phi_N|}\sum_{n\in\Phi_N}\delta_{T^na},
\]
where $\delta_x$ is the Dirac measure at the point $x \in X$ and the limit is taken in the weak* topology.

\begin{Theorem}
\label{thm_maindynamical}
Let $(X,\mu,T)$ be an ergodic measure preserving system.
For every $k\in\N$, every \Folner{} sequence $\Phi$, every $a\in\gen(\mu,\Phi)$, and every open set $E \subset X$ with $\mu(E)>0$, there exists an integer $t \ge 0$ and an \Erdos{} progression of the form $(a,x_1,\dots,x_k)$ with $x_i\in T^{-t}E$ for $i=1,\dots,k$.
\end{Theorem}

\begin{proof}[Proof that \cref{thm_maindynamical} implies \cref{thm main dynamic intro}]
Suppose $X,T,a,\mu, E$ are as in \cref{thm main dynamic intro}.
Using ergodic decomposition of the measure $\mu$, we can find a $T$-invariant ergodic measure $\tilde\mu$ such that $\tilde\mu(E)\geq\mu(E)$ and $\tilde\mu(\overline{\{ T^n a : n \in \N \}}) = 1$.
It follows from~\cite[Proposition 3.9]{Furstenberg-book} that there exists a \Folner{} sequence $\Phi$ in $\N$ such that $a\in\gen(\tilde\mu,\Phi)$.

By \cref{thm_maindynamical}, for each $k\in\N$ there exist an \Erdos{} progression $(a,x_1,\dots,x_k)$ and an integer $t \ge 0$ such that  $x_i\in T^{-t}E$ for  $i=1,\dots,k$.
Therefore, there exists a strictly increasing sequence $c\colon \N\to\N$ such that~\eqref{eqn_intro_E-progression_condition} holds.
\end{proof}

The proof of \cref{thm_maindynamical} occupies the rest of the paper.

\section{A variation on recurrence in $X^{k+1}$}
\label{sec_shiftedRecurrence}
Starting with this section, we depart significantly from earlier approaches~\cite{MRR,Host20,KMRR,KMRR2} to similar problems.

Fix a topological system $(X,T)$.
If $(a,x_1,\dots,x_k)$ is an \Erdos{} progression and $U_j$ is a neighborhood of $x_j$ for  $1 \le j \le k$, then there is $n \in \N$ such that  $a \in T^{-n} U_1$ and  $x_j \in U_j \cap T^{-n} U_{j+1}$ for $1 \le j \le k-1$. In $X^{k+1}$, we can write this as
\[
(a,x_1,\dots,x_{k-1},x_k) \in (X \times U_1 \times \cdots \times U_{k-1} \times U_k) \cap \TT^{-n} (U_1 \times U_2 \times \cdots \times U_k \times X)
\]
where
\[
\TT=\underbrace{T\times T\times\ldots\times T}_{k+1~\text{times}}
\]
denotes the product transformation.
This leads us to consider when the intersection
\[
(X \times U_1 \times \cdots \times U_{k-1} \times U_k) \cap \TT^{-n} (U_1 \times U_2 \times \cdots \times U_k \times X)
\]
is nonempty, motivating our next definition.

\begin{Definition}
\label{def_shifted_recurrence}
Let $(X,T)$ be a topological system and $k\in\N$.
We say that a probability measure $\tau \in \M(X^{k+1})$ is \emph{progressive} if for all open sets $U_1,\dots,U_k \subset X$ with
\[
\tau(X\times U_1 \times \cdots \times U_k)>0
\]
there exist infinitely many $n\in\N$ such that
\begin{equation}
\label{eq_shiftedrecurrence}
\tau\big((X\times U_1 \times \cdots \times U_{k-1} \times U_k)\cap \TT^{-n} (U_1 \times U_2 \times \cdots \times U_k \times X) \big)>0.
\end{equation}
\end{Definition}
The notion of a measure being \emph{progressive} depends on the value of $k$ but we omit this from our terminology.

Though we view this as a form of recurrence, it differs from standard notions in several ways.
For example, the measure $\tau$ is not assumed to be $\TT$ invariant, and~\eqref{eq_shiftedrecurrence} is not of the form $\mu(B \cap S^{-n} B$).
While usually there is no distinction between requiring a single return to the set and infinitely many returns, in our setting of distinct sets $U_i$, one can construct choices of the sets showing that the two notions are different.

The next proposition shows that a progressive probability measure produces \Erdos{} progressions.

\begin{Proposition}
\label{prop_progs_from_shifted}
Fix $k \in \N$ and let $(X,T)$ be a topological system.  Let $a \in X$, let $U_1,\dots,U_k\subset X$ be open sets,  and let $\tau \in \M(X^{k+1})$ be a progressive measure satisfying $\tau(\{a\} \times X^k) = 1$.
If $\tau(X \times U_1 \times \cdots \times U_k) > 0$, then there is an \Erdos{} progression $(a,x_1,\dots,x_k)$ with $x_j \in U_j$ for all $1 \le j \le k$.
\end{Proposition}

\begin{proof}
Let $V=U_1 \times \cdots \times U_k$.
Using the assumption that $\tau$ is progressive,  we can find $c(1) \in \N$ such that
\[
\tau \big( (X \times V) \cap \TT^{-c(1)} (V \times X) \big) > 0.
\]
Since $V$ is open, it follows that the intersection above contains some point $(v_0,v_1,\dots,v_k)$ in the support of $\tau$.
From the assumption that the closed set $\{a\}\times X^k$ has full measure, it follows that $v_0=a$.
Letting $V_1\subset X^k$ be an open ball centered at $(v_1,\dots,v_k)$ with a sufficiently small radius, we have all the following properties.
\begin{itemize}
\item
$\tau( \{a\} \times V_1) > 0$
\item
$\operatorname{diam}(V_1) \le \frac12\operatorname{diam}(V)$
\item
$\{a\}\times \overline{V_1}\subset(\{a\}\times V)\cap \TT^{-c(1)} (V \times X)$
\end{itemize}
We proceed inductively, setting $V_0=V$ and constructing, for each $n \in \N$, an open set $V_{n}\subset X^k$ and a recurrence time $c(n)\in\N$ satisfying the following properties:
\begin{enumerate}[label=(\roman{enumi}),ref=(\roman{enumi}),leftmargin=*]
\item
\label{itm_pr_1}
$\displaystyle \overline{V_n}\subset V_{n-1}\qquad\text{ and }\qquad\operatorname{diam}(V_n) \le \frac12\operatorname{diam}(V_{n-1})$;
\item
\label{itm_pr_2}
$\tau( \{a\}  \times V_n)>0$;
\item
\label{itm_pr_4}
$\{a\}\times \overline{V_n} \subset \TT^{-c(n)} (V_{n-1}\times X)$;
\item
$c(n+1) > c(n)$.
\end{enumerate}
The case $n=1$ is already established.
The inductive step is the same: using~\ref{itm_pr_2} and the hypotheses on $\tau$ we can find $c(n+1) > c(n)$ for which
\[
(\{a\} \times V_n) \cap \TT^{-c(n+1)} (V_n \times X)
\]
has positive $\tau$ measure.
Since $V_n$ is open, it follows that there exists a point in this intersection that belongs to the support of $\tau$.
By intersecting a small ball around that point with $\{a\}\times X^k$, we can find an open set $V_{n+1}\subset X^k$ satisfying properties~\ref{itm_pr_1}--\ref{itm_pr_4}.

We now produce the desired \Erdos{} progression.
Using property~\ref{itm_pr_1}, the intersection
\[
\bigcap_{n=1}^\infty \{a\} \times \overline{V_n}
\]
is a singleton.  Calling this point $(a,x_1,\dots,x_k)$, then property~\ref{itm_pr_4} gives that
\[
\TT^{c(n)} (a,x_1,\dots,x_{k-1})
\in
V_{n-1}
\]
for all $n \in \N$.
It thus follows that
\[
\lim_{n \to \infty} \TT^{c(n)}(a,x_1,\dots,x_{k-1}) = (x_1,x_2,\dots,x_k),
\]
as desired.
\end{proof}

We remark that the \Erdos{} progression found in the proof of \cref{prop_progs_from_shifted} is in the support of $\tau$. Therefore it follows that the support of $\tau$ contains a dense set of \Erdos{} progressions.
In our earlier work~\cite{KMRR2} we constructed a measure which was not only progressive, but had the property that almost every $x\in X^3$ is an \Erdos{} progression.
Due to the reliance on \cref{prop_progs_from_shifted} in the current work, we do not know if the measure we produce has this stronger property.
It would be interesting to know whether the conclusion of \cref{prop_progs_from_shifted} can be strengthened to match earlier work.

Using \cref{prop_progs_from_shifted}, we can now derive \cref{thm_maindynamical} from the following result.

\begin{Theorem}
\label{thm_shifted_recurrence_happens}
Let $(X,\mu,T)$ be an ergodic measure preserving system.
For every $k \in \N$, every \Folner{} sequence $\Phi$, every $a \in \gen(\mu,\Phi)$, and every open set $E \subset X$ with $\mu(E) > 0$, there exists an integer $t \ge 0$ and a progressive measure $\sigma \in \M(X^{k+1})$ such that
\begin{equation}
\label{eqn:sigma_assumption}
\sigma(X \times T^{-t} E \times \cdots \times T^{-t} E) > 0
\end{equation}
and $\sigma(\{a\} \times X^k) = 1$.
\end{Theorem}

\begin{proof}[Proof that \cref{thm_shifted_recurrence_happens} implies \cref{thm_maindynamical}]
We apply \cref{prop_progs_from_shifted} with $U_j = T^{-t} E$ for all $1 \le j \le k$ and $\tau = \sigma$.
The hypotheses of \cref{thm_shifted_recurrence_happens} allow us to conclude from \cref{prop_progs_from_shifted} that there is an \Erdos{} progression $(a,x_1,\dots,x_k)$ with $x_j \in U_j$ for all $1 \le j \le k$.
\end{proof}

\section{Constructing a progressive  measure}
\label{sec_constructingsigma}

\subsection{Two motivating examples}

To motivate our approach to constructing a progressive measure, we consider two special cases that represent complementary behaviors: irrational rotations and weakly mixing systems.

\begin{example}
\label{eg:tau rotation}
Let $X=\T=\R/\Z$ be the unit circle and let $\alpha\in\R$ be irrational. The map  $T\colon X\to X$ given by $Tx = x+\alpha\mod  1$  is ergodic with respect to the Haar measure on $\T$.
In this case, $(a,x_1,\dots,x_k)$ is an \Erdos{} progression if and only it is an arithmetic progression in $\T$, meaning that
\[
(a,x_1,\dots,x_k) = (a, a + \beta, \dots, a + k\beta)
\]
for some $\beta \in\T$.
In particular, the set of all \Erdos{} progressions starting at $a$ is a subtorus of $\T^{k+1}$.
Writing $\sigma$ for the Haar measure on this subtorus, we can describe $\sigma$ dynamically via the formula
\[
\sigma = \lim_{N \to \infty} \dfrac{1}{N} \sum_{n=1}^N \delta_a \times \delta_{T^na} \times \dots \times \delta_{T^{kn}a}
\]
with the limit taken in the weak* topology.
Open sets $U_1,\dots,U_k\subset\T$ satisfy
\[
\sigma(\T\times U_1\times\cdots\times U_k)>0
\]
precisely when there exists $\beta\in\T$ with $a+i\beta\in U_i$ for  $i=1,\dots,k$.
In this case, choosing an integer $n$ such  that $n\alpha\approx\beta$, we have that $a+i\beta\in U_i\cap T^{-n}U_{i+1}$ for all $i= 1,\dots,k-1$.
This in turn implies that
\[
\sigma\Big(\big(\T\times U_1\times\cdots\times U_k\big)\cap \TT^{-n}\big(U_1\times\cdots\times U_k\times\T\big)\Big)>0
\]
showing that $\sigma$ is progressive.
\end{example}

In view of this example, a natural candidate for $\sigma$ in the general setting is
\begin{equation}
\label{eq_naivesigma}
\sigma = \lim_{N \to \infty} \dfrac{1}{|\Phi_N|} \sum_{n \in \Phi_N} \delta_a \times \delta_{T^na} \times \cdots \times \delta_{T^{kn}a}
\end{equation}
where $\Phi$ is a \Folner{} sequence along which the limit exists.
where one would pass to a subsequence of $(\Phi_N)_{N \in \N}$ if necessary to guarantee the existence of the limit.
Unfortunately, even in the complementary case of weakly mixing systems, we are unable to prove that the measure defined by~\eqref{eq_naivesigma} is progressive.
It would be interesting to know whether it always is.
This makes it necessary for us to take a different approach for weakly mixing systems, illustrated by the following example.

\begin{example}\label{eg:tau_weak_mixing}
If $(X,\mu,T)$ is weakly mixing, then we can take $\sigma$ to be the product measure $\sigma=\delta_a\times\mu\times\cdots\times\mu\in\M(X^{k+1})$.
To verify that $\sigma$ is progressive, fix $U_1,\dots,U_k\subset X$ with $\mu(U_i)>0$ for each $i$.
Weak mixing implies each of the sets $\{n\in\N:\mu(U_i\cap T^{-n}U_{i+1})>0\}$ has full density with respect to every \Folner{} sequence.
Setting $U_0=U_{k+1}:=X$ for convenience, we conclude that
\[
\begin{split}
\big\{ n \in \N : \sigma\bigl((U_0\times\cdots\times U_k)\cap\TT^{-n}(U_1\times\cdots & \times U_{k+1})\bigr) > 0 \big\}
\\
&\supset
\bigcap_{i=0}^k
\{ n \in \N : \mu(U_i\cap T^{-n}U_{i+1}) > 0 \}
\end{split}
\]
has full density with respect to every \Folner{} sequence.
This shows that $\sigma$ is progressive and hence satisfies the conclusion of
\cref{thm_shifted_recurrence_happens}.
In fact, one can prove that for $\mu\times\cdots\times\mu$-almost every $(x_1,\ldots,x_k)\in X^k$ the point $(a,x_1,\ldots,x_k)$ is an \Erdos{} progression;
when $\mu$ is weakly mixing with respect to $T$, it follows that $\mu\times\cdots\times\mu$ is ergodic with respect to $\TT$ and hence almost every point is both generic for and in the support of this measure.
\end{example}

To define $\sigma$ in general, we construct a measure that combines structured behavior, like that of \cref{eg:tau rotation}, with mixing behavior, like that of \cref{eg:tau_weak_mixing}. To make this precise, we use the notion of pronilfactors.

\subsection{Pronilfactors}
Let $(X,\mu,T)$ and $(Y,\nu,S)$ be measure preserving systems. We say that $(Y,\nu,S)$ is a \define{factor} of $(X,\mu,T)$ if there exists an almost surely defined and measurable map $\pi\colon X\to Y$ called the \define{factor map} such that $\pi\circ T=S\circ \pi$ holds $\mu$-almost everywhere and $\pi\mu=\nu$.
Of particular importance for us are the factors which are nilsystems.

\begin{Definition}[Nilsystems, pronilsystems]
\label{def:nil}
Let $G$ be an $s$-step nilpotent Lie group and let $\Gamma\subset G$ be a discrete and cocompact subgroup.
The compact manifold $X = G/\Gamma$ is an \emph{$s$-step nilmanifold}.

The group $G$ acts on $X$ by left translation, and there is a unique Borel probability measure $\mu$ on $X$ that is invariant under this action, the \define{Haar measure} on $X$.
Letting $T\colon X\to X$ denote left translation by a fixed element of $G$, the resulting measure preserving system $(X, \mu, T)$ is called an \emph{$s$-step nilsystem} (as usual endowed with the Borel $\sigma$-algebra).

An inverse limit of $s$-step nilsystems is called an \emph{$s$-step pronilsystem.}
\end{Definition}

We note that we can view the inverse limit in the category of measure preserving systems or in the category of topological dynamical systems, as both  describe the same system.

A topological system $(X,T)$ is \define{uniquely ergodic} if there is a unique $T$-invariant measure. A pronilsystem is uniquely ergodic if and only if every point is generic for $\mu$ along every \Folner{} sequence.
It is a classical result that for all $s$-step pronilsystems $(X,\mu, T)$ the properties (i) minimal; (ii) transitive; (iii) ergodic; (iv) uniquely ergodic; are equivalent.
(See~\cite[Theorem~11, Chapter~11]{HK-book} for a summary and discussion.)
Moreover, every orbit closure in a pronilsystem supports a unique invariant measure (see~\cite{lesigne-91, leibman-05}).

For each $s\in\N$, every ergodic system $(X,\mu,T)$ has a maximal factor that is isomorphic to an $s$-step pronilsystem (see~\cite[Chapter 16]{HK-book}); this system is called  the \define{$s$-step pronilfactor} of $(X,\mu,T)$ and is denoted $(Z_s,m_s,T)$.
By a standard abuse of notation we use $T$ to denote the transformation not only in the system but also in the pronilfactor.

\subsection{Constructing the measure $\sigma$}
\label{sec_def_sigma}

We are now ready to construct the probability measure $\sigma$ on $X^{k+1}$ that we later show to be progressive.
This is carried out in two steps: first we define a progressive measure on $Z^{k+1}_{k-1}$ where $(Z_{k-1},m_{k-1},T)$ is the $(k-1)$-step pronilfactor of the given system $(X,\mu,T)$, and then we lift this to a measure on $X^{k+1}$.

For the first step, when $(X,\mu,T)$ is a pronilsystem, the measure $\sigma$ described in~\eqref{eq_naivesigma} is progressive; this follows from \cref{thm_sigma_shifted} below but a direct proof is more involved than the special case in \cref{eg:tau rotation} of a rotation.

For the second step, we need to overcome the obstacle that the definition of $\sigma$ in~\eqref{eq_naivesigma} depends on the single point $a\in\gen(\mu,\Phi)$, but the map $\pi_s\colon X\to Z_s$ to the $s$-step pronilfactor is, a priori, only a measurable map and only defined almost everywhere.
To remedy this issue, we restrict attention to systems where the map $\pi$ is continuous.

\begin{Definition}
\label{def:toppronil}
We say that $(X, \mu, T)$ \define{has topological pronilfactors} if for every $s\in\N$, there is a continuous factor map $\pi_s\colon X\to Z_s$, where $(Z_s, m_s, T)$ denotes the $s$-step pronilfactor of $(X, \mu, T)$.
\end{Definition}
Fortunately, it is possible to reduce \cref{thm_shifted_recurrence_happens} to the case when $(X,\mu,T)$ has topological pronilfactors.
This maneuver has been used in related situations, such as~\cite{HK09}, and was first adapted to this setting by Host~\cite{Host20}.
See~\cite[Section 5]{KMRR} for a more complete discussion. The reduction of \cref{thm_shifted_recurrence_happens} to the case of topological pronilfactors is proven in the next section

If $(Y,\nu,S)$ is a factor of $(X,\mu,T)$ with factor map $\pi\colon X\to Y$, then the \define{conditional expectation} of a function of $f\in L^2(\mu)$ onto $Y$, denoted $\E(f \mid Y)$, is the unique function in $L^2(\nu)$ such that
\[
\int_{\pi^{-1}(B)} f\d\mu = \int_B \E(f \mid Y)\d\nu
\]
for all Borel measurable sets $B\subset Y$.

Using this, we can define the measure $\sigma$.

\begin{Definition}
\label{def_tausigma}
Let $(X,\mu,T)$ be an ergodic system with topological pronilfactors and fix $k\in\N$.
Let $(Z_{k-1},m_{k-1},T)$ denote the $(k-1)$-step pronilfactor of $(X,\mu,T)$ and let $\pi_{k-1}\colon X\to Z_{k-1}$ denote the corresponding continuous factor map.
Given a \Folner{} sequence $(\Phi_N)_{N \in \N}$ and $a \in \gen(\mu,\Phi)$ define the measure $\xi\in\M(Z_{k-1}^k)$ by setting
\begin{equation}
\label{def:xi}
\xi
=
\lim_{N \to \infty} \dfrac{1}{|\Phi_N|} \sum_{n \in \Phi_N}
\delta_{T^n(\pi_{k-1} a)}\otimes\delta_{T^{2n}(\pi_{k-1} a)}\otimes\cdots\otimes\delta_{T^{kn}(\pi_{k-1} a)}.
\end{equation}
We then define the measure $\sigma\in\M(X^{k+1})$ via the formula
\begin{equation}
\label{define:sigma}
\int_{X^{k+1}} f_0 \otimes f_1 \otimes\cdots\otimes f_k\d\sigma
=
f_0(a)\int_{Z_{k-1}^{k}}\E(f_1\mid Z_{k-1})\otimes\dots\otimes\E(f_k\mid Z_{k-1})\d\xi
\end{equation}
for all continuous functions $f_0,f_1,\dots,f_k \colon X \to \C$.
\end{Definition}

The limit in~\eqref{def:xi} exists and is in fact independent of $\Phi$, since $(Z_{k-1}^k,T\times T^2\times\cdots\times T^k)$ is a topological pronilsystem and hence every point has a uniquely ergodic orbit closure.
It follows from this that
$\xi$ is the unique measure on $\M(Z_{k-1}^k)$ such that $(\pi_{k-1}(a),\dots,\pi_{k-1}(a))$ is a generic point (along $\Phi$) with respect to the transformation $T\times T^2\times\cdots\times T^k$

We stress that the measures $\xi$ and $\sigma$ depend on the choice of the point $a$, as well as the system $(X,\mu,T)$ and the positive integer $k$.
For ease of reading, we omit these dependencies from the notation.
Note also that while $\xi$ is defined on a $k$-fold product space indexed by $\{1,\dots,k\}$,  the measure $\sigma$ is defined on a $(k+1)$-fold product space indexed by $\{0,1,\dots,k\}$ and distinguishes the zeroth coordinate.

For a measure $\tau\in\M(X^k)$ and each $i = 1, \dots ,k$, we let $\tau_i$ denote the \emph{$i$-th marginal} of $\tau$, meaning the projection of $\tau$ onto the $i$-th coordinate.
Given two measures $\tau, \tau'$ defined on the same space, we say that $\tau\leq \tau'$ if $\tau(A)\leq \tau'(A)$ for all measurable sets $A$.

\begin{Lemma}
\label{lem_sigma_marginals}
Fix an ergodic system $(X,\mu,T)$ with topological pronilfactors, an integer $k \in \N$, a \Folner{} sequence $(\Phi_N)_{N \in \N}$, and a point $a \in \gen(\mu,\Phi)$.
For every $1 \le i \le k$, the $i$-th marginal $\sigma_i$ of the measure $\sigma$ in \cref{def_tausigma} satisfies $\sigma_i \le i \mu$.
\end{Lemma}

\begin{proof}
Note that $\xi_i$, the $i$-th marginal of the measure $\xi$ from~\eqref{def:xi}, is $T^i$-invariant.  Since the system $(Z_{k-1},T)$ is uniquely ergodic, we have that  $\xi_i\leq i m_{Z_{k-1}}$ and hence
\begin{equation}
\label{eqn_ioxi}
\int_{Z_{k-1}}g\d\xi_i\leq i\int_{Z_{k-1}}g\d m_{Z_{k-1}}
\end{equation}
whenever $g\colon Z_{k-1}\to[0,1]$ is Borel measurable.
Therefore, for any Borel set $A\subset X$,  we have
\[\sigma_i(A)
=
\int_{Z_{k-1}}\E(\one_A\mid Z_{k-1})\d\xi_i
\leq
i\int_{Z_{k-1}}\E(\one_A\mid Z_{k-1})\d m_{Z_{k-1}}
=
i\mu(A).\quad\qedhere
\]
\end{proof}

We also record the following which follows directly from the definition. \begin{lemma}\label{lemma_sigmainvariance}
The measure $\sigma$ defined in \cref{def_tausigma} is invariant under $\Id\times T\times T^2\times\cdots\times T^k$.
\end{lemma}

\subsection{Outlining the proof of \cref{thm_shifted_recurrence_happens}}

To prove \cref{thm_shifted_recurrence_happens}, we need to verify that $\sigma$ has the required properties, which are captured by the following two theorems.

\begin{Theorem}
\label{thm_sigma_sz_thing}
Fix an ergodic system $(X,\mu,T)$ with topological pronilfactors.
Let $k \in \N$, let $(\Phi_N)_{N \in \N}$ be a \Folner{} sequence, let $a \in \gen(\mu,\Phi)$, and let $\sigma$ be the measure described in \cref{def_tausigma}.
Then for every open set $E \subset X$ with $\mu(E) > 0$, there is an integer $t \ge 0$ such that  $\sigma(X \times T^{-t} E \times \cdots \times T^{-t} E) > 0$.
\end{Theorem}

\begin{Theorem}
\label{thm_sigma_shifted}
Fix an ergodic system $(X,\mu,T)$ with topological pronilfactors.
Let $k \in \N$, let $(\Phi_N)_{N \in \N}$ be a \Folner{} sequence, and let $a \in \gen(\mu,\Phi)$.
Then the measure $\sigma$ described in \cref{def_tausigma} is progressive.
\end{Theorem}

These ingredients, together with the ability to pass to an extension with topological pronilfactors, combine to give a quick proof of \cref{thm_shifted_recurrence_happens}.

\begin{proof}[Proof that Theorems~\ref{thm_sigma_sz_thing} and~\ref{thm_sigma_shifted} together imply \cref{thm_shifted_recurrence_happens}]
Fix an ergodic system $(X,\mu,T)$, a \Folner{} sequence $(\Phi_N)_{N \in \N}$, a point $a \in \gen(\mu,\Phi)$, and an open set $E \subset X$ with $\mu(E) > 0$.
By~\cite[Proposition 5.7 and Lemma 5.8]{KMRR}, there exist an ergodic system $(\tilde X,\tilde \mu,\tilde T)$ with topological pronilfactors and a continuous factor map $\pi \colon \tilde X \to X$ such that
$(X,\mu,T)$ is a factor of $(\tilde X,\tilde \mu,\tilde T)$.
Moreover, there is a \Folner{} sequence $(\tilde \Phi_N)_{N \in \N}$ and $\tilde a \in \gen(\tilde \mu,\tilde \Phi)$ such that $\pi(\tilde a) = a$.
Let $\tilde\sigma$ denote the measure on $(\tilde X)^{k+1}$ defined by~\eqref{define:sigma} and let $\tilde E:=\pi^{-1}(E)$.
Applying Theorems~\ref{thm_sigma_sz_thing} and~\ref{thm_sigma_shifted} to $\tilde \sigma$, it follows that $\tilde\sigma$
is a progressive measure, that  $\tilde{\sigma}( \{ \tilde{a} \} \times X^k ) = 1$, and that there exists an integer $t \ge 0$ satisfying $\tilde\sigma(\tilde X\times \tilde T^{-t}\tilde E\times\cdots\times \tilde T^{-t}\tilde E)>0$.

Letting $\sigma$ to be the push-forward of $\tilde \sigma$ under the map $\pi \times \pi \times \dots \times \pi$, it follows that $\sigma$ satisfies the conclusion of \cref{thm_shifted_recurrence_happens}.
\end{proof}

\section{Proof of \cref{thm_sigma_sz_thing}}
\label{sec_findingt}

We are left with proving Theorems~\ref{thm_sigma_sz_thing} and~\ref{thm_sigma_shifted}.
The proof of the former is given in this section, whereas the proof of the latter is in \cref{sec_proof_of_sigma_shifted}.
In both proofs we relate certain dynamical averages of $\sigma$ to the Furstenberg joining of $(X,\mu,T)$.

\begin{Definition}[The Furstenberg joining]
\label{def:Furstjoining}
Given a measure preserving system $(X,\mu,T)$ and $k\in\N$, let $\lambda$ be the Borel probability measure on $X^{k+1}$ uniquely determined by
\begin{align}
\label{defn:furstenberg_joining}
\int_{X^{k+1}} f_0\otimes f_1\otimes \cdots\otimes f_k \d\lambda = \lim_{N\to\infty}\frac{1}{|\Phi_N|}\sum_{n\in\Phi_N} \int_X f_0\cdot T^nf_1\cdots T^{kn}f_k \d\mu
\end{align}
for all $f_0,f_1,\ldots,f_k\in L^\infty(\mu)$.
The measure $\lambda$ determined this way is called the \define{Furstenberg joining} of $(X,\mu,T)$.
\end{Definition}

This joining is well defined, as it is shown in~\cite{HK05} that the limit exists and is independent of the choice of \Folner{} sequence $\Phi=(\Phi_N)_{N\in\N}$.

As we show, \cref{thm_sigma_sz_thing} follows quite quickly from the following theorem and Furstenberg's multiple recurrence theorem.

\begin{Theorem}
\label{sigma_diagonal_average_furstenberg}
Fix an ergodic measure preserving system $(X,\mu,T)$ with topological pronilfactors, $k \in \N$, a \Folner{} sequence $(\Phi_N)_{N \in \N}$ and $a \in \gen(\mu,\Phi)$.
Let $\sigma$ be the measure described in \cref{def_tausigma}.
For any $g_1,\dots,g_k \in C(X)$ we have
\begin{equation}
\label{eqn_sigma_average_furstenberg}
\lim_{N\to\infty} \frac{1}{|\Phi_N|} \sum_{n\in\Phi_N} \int_{X^{k+1}}\TT^n(\one\otimes g_1 \otimes \cdots \otimes g_k)\d\sigma
=
\int_{X^{k+1}} \one \otimes g_1 \otimes \cdots \otimes g_k \d\lambda
\end{equation}
where $\lambda$ is the Furstenberg joining.
\end{Theorem}

\begin{proof}[Proof that \cref{sigma_diagonal_average_furstenberg} implies \cref{thm_sigma_sz_thing}]
Since $\mu$ is a Borel measure on the compact metric space $X$, it is regular. Therefore, given an open set $E\subset X$ with $\mu(E)>0$, one can use Urysohn's lemma to find $g \in C(X)$ with $\int g\d\mu>0$ and $0 \le g \le \one_E$.
Thus
\[
\int_{X^{k+1}}\TT^n(\one \otimes \one_E \otimes \cdots \otimes \one_E)\d\sigma
\ge
\int_{X^{k+1}}\TT^n(\one\otimes g \otimes \cdots \otimes g)\d\sigma
\]
for all $n \in \N$.
Combining this observation with \cref{sigma_diagonal_average_furstenberg} gives that
\begin{eqnarray*}
\lim_{N\to\infty} \frac{1}{|\Phi_N|} \sum_{n\in\Phi_N} \int_{X^{k+1}}\TT^n(\one\otimes \one_E \otimes \cdots \otimes \one_E)\d\sigma
&\geq&
\int \one \otimes g \otimes \cdots \otimes g \intd \lambda.
\end{eqnarray*}
By the definition of $\lambda$, this last expression is equal to
\[
\lim_{N\to\infty}\frac1{|\Phi_N|}\sum_{n\in\Phi_N}\int_XT^ng\cdots T^{kn}g\d\mu
\]
and this is positive by Furstenberg's multiple recurrence theorem~\cite[Theorem 11.13]{Furstenberg-1977}.
We have shown that
\[
\liminf_{N \to \infty} \dfrac{1}{|\Phi_N|} \sum_{n \in \Phi_N} \sigma(X \times T^{-t} E \times \cdots \times T^{-t} E) > 0,
\]
which is a stronger statement than what is required in the conclusion of \cref{thm_sigma_sz_thing}.
\end{proof}

It remains to prove \cref{sigma_diagonal_average_furstenberg}.

\begin{proof}[Proof of \cref{sigma_diagonal_average_furstenberg}]
Fix $g_1,\dots,g_k \in C(X)$.
Our goal is to establish~\eqref{eqn_sigma_average_furstenberg}.
It follows from~\cite[Theorem 12.1]{HK05} that
\[
\int g_1 \otimes \cdots \otimes g_k \intd \lambda
=
\lim_{L\to\infty}\frac{1}{|\Phi_L|}\sum_{\ell\in\Phi_L} \int_{Z_{k-1}} T^\ell \E(g_1\mid Z_{k-1}) \cdots T^{k\ell} \E(g_k\mid Z_{k-1})\d m_{k-1}
\]
where $(Z_{k-1}, m_{k-1}, T)$ is the $(k-1)$-step pronilfactor of $(X,\mu,T)$.
Combined with~\eqref{define:sigma} it follows that, in order to establish~\eqref{eqn_sigma_average_furstenberg}, it suffices to prove
\[
\lim_{N\to\infty}\frac{1}{|\Phi_N|}\sum_{n\in\Phi_N} \int_{Z_{k-1}^{k}} T^n h_1\otimes\cdots\otimes T^n h_k\d\xi
=
\lim_{L\to\infty}\frac{1}{|\Phi_L|}\sum_{\ell\in\Phi_L}\int_{Z_{k-1}} T^{\ell}h_1\cdots T^{k\ell}h_k\d m_{k-1}
\]
whenever $h_1,\dots,h_k : Z_{k-1} \to \C$ are bounded and measurable.
By a standard approximation argument and \cref{lem_sigma_marginals} it suffices to establish this when the functions $h_1,\dots,h_k$ are continuous, which we henceforth assume.
Writing $\tilde{a}=\pi_{k-1}(a)$ we have the following from the definition of $\xi$:
\begin{equation}
\label{eq_proof_first_main_theorem_about_sigma1}
\begin{split}
\lim_{N\to\infty}\frac1{|\Phi_N|}\sum_{n\in\Phi_N} & \int_{Z_{k-1}^{k}} T^n h_1\otimes\cdots\otimes T^n h_k\d\xi
\\
&=
\lim_{N\to\infty}\lim_{L\to\infty}\frac1{|\Phi_N|}\sum_{n\in\Phi_N}\frac1{|\Phi_L|}\sum_{\ell\in\Phi_L} h_1(T^{n+\ell}\tilde{a}) \cdots h_k(T^{n+k\ell}\tilde{a}).
\end{split}
\end{equation}
On the other hand, ergodicity of $(X,\mu,T)$ ensures ergodicity, and hence unique ergodicity, of the pronilfactor $(Z_{k-1},m_{k-1},T)$, and so it follows that
\begin{equation}
\label{eq_proof_first_main_theorem_about_sigma2}
\begin{split}
\lim_{L\to\infty}\frac{1}{L}\sum_{\ell=1}^L & \int_{Z_{k-1}} T^{\ell}h_1\cdots T^{k\ell}h_k\d m_{k-1}
\\
&=
\lim_{L\to\infty}\lim_{N\to\infty}
\frac1{|\Phi_L|}\sum_{\ell\in\Phi_L}
\frac1{|\Phi_N|}\sum_{n\in\Phi_N}
h_1(T^{\ell+n}\tilde{a})\cdots h_k(T^{k\ell+n}\tilde{a}).
\end{split}
\end{equation}
Consider now the commuting nilrotations
\begin{gather*}
S = T \times T \times \cdots \times T \\
R = T \times T^2 \times \cdots \times T^k
\end{gather*}
on $(Z_{k-1})^k$.
Together they induce an action of $\Z^2$ on $Z_{k-1}^k$.
Set
\[
Y:=\overline{\{S^nR^\ell(\tilde a,\dots,\tilde a):n,\ell\in\Z\}}\subset Z_{k-1}^k
\]
to be the orbit closure of $(\tilde a,\dots,\tilde a)\in\Z_{k-1}^k$ under $S$ and $R$.  Since $(Z_{k-1},m_{k-1},T)$ is a pronilsystem, the $\Z^2$-system $(Y,S,R)$ is also a pronilsystem and hence is uniquely ergodic~\cite[Theorem 17, Chapter 11]{HK-book}.
Therefore, there is a unique invariant mean for these averages, which implies (see for example~\cite[Lemma 1.1]{bergelson-leibman15}) that the expressions in~\eqref{eq_proof_first_main_theorem_about_sigma1} and~\eqref{eq_proof_first_main_theorem_about_sigma2} coincide, finishing the proof.
\end{proof}

\section{Proof of \cref{thm_sigma_shifted}}
\label{sec_proof_of_sigma_shifted}
\subsection{Ingredients in the proof of \cref{thm_sigma_shifted}}

We are left with proving  \cref{thm_sigma_shifted}.
The main ingredients are the following two theorems.
The first -- a consequence of Furstenberg's multiple recurrence theorem -- is proved in \cref{subsec:furstenberg_extension}.
The second is proved in \cref{sec_sigma_coordinate_invariant} and uses material about uniformity norms covered in \cref{subsec:uniformity_norms}.

\begin{Theorem}
\label{cor_thm_weirdgeneralSzemeredi}
Let $(X,T)$ be a topological system, let $\Phi =(\Phi_N)_{N\in\N}$ be a F\o lner sequence in $\N$, and let $\nu\in\M(X^{k})$ be $T\times T^2\times\cdots\times T^k$ invariant.
If $G\colon X^k\to[0,1]$ is continuous and satisfies $\int_{X^k}G\d\nu>0$, then
$$\liminf_{N\to\infty}\frac1{|\Phi_N|}\sum_{n\in\Phi_N}\int_{X^k}G\cdot \TT^nG\d\nu>0.$$
\end{Theorem}
\begin{Theorem}
\label{sigma_coordinate_invariant}
\label{second_main_theorem_about_sigma}
Fix an ergodic measure preserving system $(X,\mu,T)$ with topological pronilfactors, $k \in \N$, a \Folner{} sequence $(\Phi_N)_{N \in \N}$, and a point $a \in \gen(\mu,\Phi)$.
Let $\sigma$ be the measure given in \cref{def_tausigma}.
For any continuous function $G \colon X^k \to \C$ we have
\[
\lim_{N\to\infty}\bigg\|\frac1{|\Phi_N|}\sum_{n\in\Phi_N} \TT^n (G\otimes\one)\,-\, \frac1{|\Phi_N|}\sum_{n\in\Phi_N} \TT^n(\one\otimes G)\bigg\|_{L^2(\sigma)}=0.
\]
\end{Theorem}

Using these two theorems we can quickly prove \cref{thm_sigma_shifted}.

\begin{proof}[Proof that Theorems \ref{cor_thm_weirdgeneralSzemeredi} and \ref{second_main_theorem_about_sigma} imply \cref{thm_sigma_shifted}]
Our goal is to show that the measure $\sigma$ defined in \cref{def_tausigma} is progressive.
Fix open sets $U_1,\dots,U_k$ with
\[
\sigma(X \times U_1 \times \dots \times U_k) > 0
\]
and set $V = U_1 \times \dots \times U_k$.
Since $V$ is open and $\sigma(X\times V)>0$, there exists
a continuous function $G\in C(X^{k})$ such that $0\leq G\leq \one_{V}$ and
$$\int_{X^{k+1}}1\otimes G\d\sigma>0.$$
Applying \cref{second_main_theorem_about_sigma}, we obtain
\begin{align*}
\liminf_{N\to\infty}\frac1{|\Phi_N|} & \sum_{n\in\Phi_N}\sigma\big((X\times V)\cap \TT^{-n} (V \times X)\big)
\\
\geq
{}&
\liminf_{N\to\infty}\frac1{|\Phi_N|}\sum_{n\in\Phi_N}\int_{X^{k+1}}\big(\one\otimes G\big)\cdot\TT^n\big(G\otimes\one\big)\d\sigma
\\
=
{}&
\liminf_{N\to\infty}\frac1{|\Phi_N|}\sum_{n\in\Phi_N}\int_{X^{k+1}}\big(\one\otimes G\big)\cdot\TT^n\big(\one\otimes G\big)\d\sigma.
\end{align*}
Let $\nu\in\M(X^k)$ denote the projection of $\sigma$ to the last $k$ coordinates.
In view of \cref{lemma_sigmainvariance}, $\nu$
is invariant under $T\times T^2\times\cdots\times T^k$ and hence we can use \cref{cor_thm_weirdgeneralSzemeredi} to conclude that
the last expression is positive.
This shows that $\sigma$ is progressive and completes the proof.
\end{proof}

\subsection{An extension of Furstenberg's multiple recurrence theorem}
\label{subsec:furstenberg_extension}
We derive \cref{cor_thm_weirdgeneralSzemeredi} from the following more general result, which is an
extension of Furstenberg's multiple recurrence theorem that  may be of independent interest.

Given a vector $v=(v_1,\dots,v_k)\in\N^{k}$, we write  $T_v=T^{v_1}\times\cdots\times T^{v_k}$.

\begin{Theorem}
\label{thm_weirdgeneralSzemeredi}
    Let $(X,T)$ be a topological system, let $u,v\in\N^k$, and let $\Phi=(\Phi_N)_{N\in\N}$ be a \Folner{}-sequence in $\N$.
    If $\nu\in\M(X^k)$ is $T_v$-invariant and $A_1,\dots,A_k\subset X$ are such that $A=A_1\times\cdots\times A_k\subset X^k$ satisfies $\nu(A)>0$, then
    \begin{equation}
    \label{eqn_weirdmultirecurrence}
    \liminf_{N\to\infty}\frac1{|\Phi_N|}\sum_{n\in\Phi_N} \nu(A\cap T_u^{-n}A)>0.
    \end{equation}
\end{Theorem}
When $u=v$, the conclusion of \cref{thm_weirdgeneralSzemeredi} follows quickly from the mean ergodic theorem.

The next example shows that in general one cannot drop the assumption that $A$ is a product set in \cref{thm_weirdgeneralSzemeredi}.

\begin{example}
    Let $T\colon\T\to\T$ be an irrational rotation $T\colon x\mapsto x+\alpha$ and let $\nu$ denote the Haar measure on the subtorus $H:=\{(x,2x):x\in\T\}\subset\T^2$.
    With $v=(1,2)$ and $u=(2,1)$ we have that $\nu$ is $T_v$-invariant, but $T_u^{-n}H=H+(0,3n\alpha)$ is disjoint from $H$ for every $n\in\N$. In particular~\eqref{eqn_weirdmultirecurrence} does not hold with $A=H$.
\end{example}

Consider the case that $(X,\mu,T)$ is a measure preserving system, $u=(1,\dots,k)$, $v=(1,\dots,1)$, and $\nu$ is the diagonal measure on $X^k$, meaning that $\nu$ is the push-forward of $\mu$ under the diagonal embedding $x\mapsto (x,x,\ldots,x)$ of $X$ into $X^k$.
Taking $A = B \times \cdots \times B$ in the theorem specializes~\eqref{eqn_weirdmultirecurrence} to
\[
\liminf_{N\to\infty}\frac1{|\Phi_N|}\sum_{n\in\Phi_N} \mu(B \cap T^{-n}B \cap \cdots\cap T^{-kn}B) >0
\]
which is Furstenberg's multiple recurrence theorem.  We do not, however, give a new proof of Furstenberg's theorem, as it is an ingredient in our proof of \cref{thm_weirdgeneralSzemeredi}.

When applying \cref{thm_weirdgeneralSzemeredi} to prove \cref{cor_thm_weirdgeneralSzemeredi}, we take $u=(1,\dots,1)$, $v=(1,\dots,k)$; this interchanges the roles of $u$ and $v$ in Furstenberg's  multiple recurrence.

\begin{proof}[Proof of \cref{thm_weirdgeneralSzemeredi}]
Define
$c=\operatorname{LCM}(v_1,\dots,v_k)$ to be the least common multiple of $v_1,\ldots,v_k$ and let $w_i=\frac{c\,u_i}{v_i}$, so that
\[
T_v^{-w_i}(X^{i-1}\times A_i\times X^{k-i})=T_u^{-c}(X^{i-1}\times A_i\times X^{k-i}).
\]
It follows that
\begin{align*}
    T_u^{-cn}A
    &=
    \bigcap_{i=1}^kT_u^{-cn}(X^{i-1}\times A_i\times X^{k-i})
    =
    \bigcap_{i=1}^kT_v^{-w_in}(X^{i-1}\times A_i\times X^{k-i})
    \\&\supset
    \bigcap_{i=i}^k T_v^{-w_in}A.
\end{align*}
It is clear that
\[
\liminf_{N\to\infty}\frac{1}{|\Phi_N|}\sum_{n\in\Phi_N} \nu(A\cap T^{-n}_uA)
\geq
\liminf_{N\to\infty}\frac{1}{|\Phi_N|} \sum_{\substack{n \in \Phi_N \\ c|n}}\nu(A\cap T^{-n}_uA)
\]
and so taking $\Psi_N=\Phi_N/c= \{n\in\N: cn\in\Phi_N\}$,
we have
\begin{align*}
\liminf_{N\to\infty}\frac1{|\Phi_N|}\sum_{n\in\Phi_N} \nu(A\cap T^{-n}_uA)
&\geq
\liminf_{N\to\infty}\frac1{c|\Psi_N|}\sum_{n\in\Psi_N}\nu(A\cap T^{-cn}_uA)
\\
&\geq
\liminf_{N\to\infty}\frac1{c|\Psi_N|}\sum_{n\in\Psi_N}\nu\left(A\cap\bigcap_{i=i}^k T_v^{-w_in}A\right)
\end{align*}
as in~\eqref{eq_longestlabelintheworld}.
Since $\nu$ is $T_v$-invariant, the last expression is positive by applying Furstenberg's multiple recurrence theorem to the system $(X^k,\nu,T_v)$.
\end{proof}

\begin{proof}[Proof of \cref{cor_thm_weirdgeneralSzemeredi}]
    Since $G\colon X^k\to[0,1]$ is continuous and has positive integral, there exist open sets $U_1,\dots,U_k$ and $c>0$ such that $c\cdot\one_{U_1}\otimes\cdots\otimes\one_{U_k}\leq G$ and $\nu(U_1\times\cdots\times U_k)>0$.
    Applying \cref{thm_weirdgeneralSzemeredi} with $u=(1,\dots,1)$, $v=(1,\dots,k)$, the conclusion follows.
\end{proof}

\subsection{Uniformity norms}
\label{subsec:uniformity_norms}

For the proof of \cref{sigma_coordinate_invariant} we make use of the structure theory of measure preserving transformations via uniformity seminorms~\cite{HK05, HK09}.
Roughly speaking, these seminorms capture the idea that functions orthogonal to the pronilfactor do not contribute to multiple ergodic averages.
We use this in the proof of \cref{second_main_theorem_about_sigma} to, roughly speaking, replace the function $G$ by a function that is measurable with respect to $Z_{k-1}^k$ without changing the $L^2(\sigma)$ norm. This reduces the proof of \cref{second_main_theorem_about_sigma} to analyzing the behavior in pronilsystems, which have additional structure.

\begin{Definition}[Uniformity Norms]
Given a measure preserving system $(X,\mu,T)$ and $s\geq 0$, the $s$-step uniformity seminorm $\| f \|_{U^s(X,\mu,T)}$ of a function $f \in L^\infty(\mu)$ is defined inductively as
\begin{gather}
\| f \|_{U^0(X,\mu,T)} = \int_X f\d\mu \\
\| f \|_{U^{s+1}(X,\mu,T)}^{2^{s+1}} = \lim_{H\to\infty}\frac{1}{H}\sum_{h=1}^H \| T^h f\cdot \overline{f} \|_{U^s(X,\mu,T)}^{2^s}. \label{eqn:semi_induc_def}
\end{gather}
The fact that the limit always exists and that $\| \cdot \|_{U^s(X,\mu,T)}$ defines a seminorm on $L^\infty(\mu)$ for $s\geq 1$ is proven in~\cite{HK05}.
\end{Definition}

The mean ergodic theorem gives
\[
\| f \|^2_{U^1(X,\mu,T)} = \| \E(f \mid \mathcal{I}) \|_2^2
\]
where $\mathcal{I}$ denotes the $\sigma$-algebra of invariant sets and so
\begin{equation}
\label{eqn:semi_1_def}
\| f \|_{U^1(X,\mu,T)} = \left| \int_X f \d \mu \right|,
\end{equation}
when the system $(X,\mu,T)$ is ergodic, in agreement with the standard definition.
The main result in~\cite{HK05} states that given an ergodic system $(X, \mu, T)$, for each $s\geq 1$, all functions $f \in L^\infty(\mu)$ satisfy
\begin{equation}\label{eq_OneLineStructureTheory}
  \| f \|_{U^{s+1}(X,\mu,T)} = 0 \iff \E(f\mid Z_s)=0,
\end{equation}
where $(Z_s,m_s,T)$ denotes the $s$-step pronilfactor of $(X,\mu,T)$.

\begin{Theorem}
\label{thm_sigma_prop_iv}
Let $(X,\mu,T)$ be ergodic, let $k \geq 2$, let $(\Phi_N)_{N \in \N}$ be a \Folner{} sequence, and let $\tau\in\M(X^{k+1})$ be invariant with respect to the transformation $\Id\times T\times T^2\times \dots \times T^k$.
Assume that the marginals $\tau_0,\tau_1,\dots,\tau_k$ of $\tau$ satisfy $\tau_i\leq i\mu$ for all $ 1 \le i \le k-1$.
Then for any $f_1, \dots, f_{k-1}\in L^\infty(\mu)$ and any bounded sequence $b\colon \N\to\C$, we have
\begin{equation}
\label{eq_thm_sigma_prop_iv}
\begin{aligned}
\limsup_{N \to \infty}
\Bigl\| \dfrac{1}{|\Phi_N|} \sum_{n \in \Phi_N} b(n) \cdot  (\one\otimes T^n f_1  &{} \otimes \dots \otimes T^n f_{k-1}\otimes \one)\Bigr\|_{L^2(\tau)}
\\
\leq {}&
C_k \cdot \|b\|_\infty\cdot \min \{ \|f_i\|_{U^k(X,\mu,T)} : 1 \le i \le k-1 \}
\end{aligned}
\end{equation}
where $C_k$ is a constant depending only on $k$.
\end{Theorem}

For the proof of \cref{thm_sigma_prop_iv}, we make use of various properties of uniformity seminorms, which we collect in the following lemma.

\begin{Lemma}
\label{lem_uniformity_seminorms_properties}
Let $(X,\mu,T)$ be a measure preserving system.
\begin{enumerate}
[label=(\roman{enumi}),ref=(\roman{enumi}),leftmargin=*]
\item \label{itm_uniformity_seminorms_properties_i}
For all $k\geq 1$ and $f\in L^\infty(\mu)$, we have
\[
\|f\otimes \overline{f}\|_{U^k(X\times X,\mu \times \mu,T\times T)}\leq
\|f\|_{U^{k+1}(X,\mu,T)}^2.
\]
\item \label{itm_uniformity_seminorms_properties_ii}
Let $c\geq 1$ be an integer.
For all $f\in L^\infty(\mu)$ and all $k\geq 1$ we have that
\begin{equation*}
\|f\|_{U^k(X,\mu,T)}\leq \|f\|_{U^k(X,\mu,T^c)},
\end{equation*}
and for all $k\geq 2$ we have that
\begin{equation}
\label{eqn_prop_seminorms_for_powers_ii2}
\|f\|_{U^k(X,\mu,T^c)}\leq c^{\frac{k}{2^k}} \|f\|_{U^k(X,\mu,T)}.    \end{equation}
\item
\label{itm_uniformity_seminorms_properties_iii}
For any $k\geq 1$ and integers $c_1,\dots,c_k \ge 1$, there is a constant $C$, independent of the system, with the following property: for all $f_1,\dots,f_k\in L^\infty(\mu)$ with $\|f_i\|_{L^\infty(\mu)}\leq1$ and all \Folner{} sequences $(\Phi_N)_{N \in \N}$, we have
\[
\lim_{N\to\infty}\Bigl\| \dfrac{1}{|\Phi_N|} \sum_{n \in \Phi_N} T^{c_1n}f_1\cdots T^{c_k n}f_k \Bigr\|_{L^2(\mu)}
\le
C
\min \{\|f_i\|_{U^{k}(X,\mu, T)} : 1\leq i\leq k \}.
\]
\item
\label{itm_uniformity_seminorms_properties_iv}
Let $k\geq2$, $c\in\N$, and let $\rho\in\M(X)$ be invariant with respect to $T^c$ and satisfy
\begin{equation}\label{eq_prop_sigma_prop_iv_ambitious2_assumption_basecase_otherlemma_n}
    \rho\leq C\mu
\end{equation}
for some constant  $C > 0$.
Then for any $f\in L^\infty(\mu)$, we have
\begin{equation}\label{eq_prop_sigma_prop_iv_ambitious2_assumption_basecase_otherlemma_n2}
\|f\|_{U^k(X,\rho,T^{c})}
\leq C^{1/2^k}\cdot c^{\frac{k}{2^k}} \|f\|_{U^k(X,\mu,T)}.
\end{equation}
\end{enumerate}
\end{Lemma}

\begin{proof}
Part~\ref{itm_uniformity_seminorms_properties_i} is proved in the appendix of~\cite[Equation (A.5)]{FH18}.
Part~\ref{itm_uniformity_seminorms_properties_ii} is a special case of~\cite[Lemma 3.1]{FrKu}.
Part~\ref{itm_uniformity_seminorms_properties_iii} follows by combining~\cite[Proposition~1]{Host}
and~\cite[Lemma 3.1]{FrKu}.
Finally, for part~\ref{itm_uniformity_seminorms_properties_iv}, using~\eqref{eq_prop_sigma_prop_iv_ambitious2_assumption_basecase_otherlemma_n} and that $\rho$ is $T^c$-invariant, it follows quickly that
\[
\|f\|_{U^k(X,\rho,T^{c})}^{2^k}\leq C\|f\|_{U^k(X,\mu,T^{c})}^{2^k}
\]
using~\eqref{eqn:semi_1_def} for the case $k = 1$ and~\eqref{eqn:semi_induc_def} to prove the cases $k > 1$ by induction.
Thus~\eqref{eq_prop_sigma_prop_iv_ambitious2_assumption_basecase_otherlemma_n2} follows from~\eqref{eqn_prop_seminorms_for_powers_ii2}.
\end{proof}

We include a short lemma to clarify how to decompose averages over \Folner{} sequences along residue classes.
\begin{lemma}\label{lemma_folnerdivision}
    Let $a(n)$ be a bounded sequence taking values in a Banach space, let $\Phi = (\Phi_N)_{N\in\N}$ be a F\o lner sequence in $\N$, let $c\in\N$, and let $\Psi = (\Psi_N)_{N\in\N}$ be the F\o lner sequence defined by $\Psi_N:=\Phi_N/c$.
    Then
    \[
    \limsup_{N\to\infty}\left\|\frac{1}{|\Phi_N|}\sum_{n\in\Phi_N}a(n)\right\|
    \leq
    \frac{1}{c} \sum_{j=0}^{c-1}\limsup_{N\to\infty}\left\|\frac{1}{|\Psi_N|}\sum_{n\in\Psi_N}a(cn+j)\right\|.
    \]
\end{lemma}

\begin{proof}
Since $\Phi$ is a \Folner{} sequence, for every $j\in\N$ we have
\[
\frac{1}{|\Phi_N|}\Big| (\Phi_N - j) \cap c\N \Big| \to \dfrac{1}{c}
\]
as $N\to\infty$.
Hence we have that
\begin{equation}
\label{eq_longestlabelintheworld}
\frac1{|\Phi_N|}\left|\Phi_N\cap \left(\bigcup_{j=0}^{c-1}c\Psi_N+j\right)\right|\to1\quad\text{ and }\quad \frac{|\Psi_N|}{|\Phi_N|}\to \frac1c
\end{equation}
as $N\to\infty$.
    It follows that as $N\to\infty$, we have that
    $$\frac1{|\Phi_N|}\left\|\sum_{n\in\Phi_N}a(n)\right\|
    -
    \frac1{|\Phi_N|}\left\|\sum_{j=0}^{c-1}\sum_{n\in\Psi_N}a(cn+j)\right\|\to0.$$
    The conclusion is obtained by using the triangle inequality.
\end{proof}

On the way to proving \cref{thm_sigma_prop_iv}, we first establish the following related result, corresponding roughly to the case $b(n)=1$ in \cref{thm_sigma_prop_iv}.

\begin{Theorem}
\label{thm_sigma_prop_V}
Let $(X,\mu,T)$ be a measure preserving system, $k \in \N$, and let $\nu\in\M(X^{k})$ be invariant under the transformation $(T\times T^2\times \dots \times T^k)$.
Assume that the marginals $\nu_i$ of $\nu$ satisfy
\begin{equation*}
\nu_i\leq i^2\mu
\end{equation*}
for all $ 1 \le i \le k$.
Then for all $f_1, \dots, f_{k}\in L^\infty(\mu)$ with $\|f_i\|_{L^\infty(\mu)}\leq1$, we have
\begin{equation}
\label{eq_thm_sigma_prop_V}
\limsup_{N \to \infty}
\Bigl\| \dfrac{1}{|\Phi_N|} \sum_{n \in \Phi_N}  (T^n f_1  \otimes \dots \otimes T^n f_{k})\Bigr\|_{L^2(\nu)}
\leq
D_k  \min \{ \|f_i\|_{U^k(X,\mu,T)} : 1 \le i \le k\}
\end{equation}
where $D_k$ is a constant depending only on $k$.
\end{Theorem}

\begin{proof}
Let $c=\operatorname{LCM}(2,\dots,k)$ be the least common multiple of the first $k$ positive integers and set $c_i = c / i$ for all $1 \le i \le k$. Let $D_k$ be the constant in part (iii) of \cref{lem_uniformity_seminorms_properties}.
Set
\[
\Psi_N = \Phi_N/c = \{ n \in \N : cn \in \Phi_N \}
\]
which defines a \Folner{} sequence.
Then, in view of \cref{lemma_folnerdivision},
\begin{align*}
\limsup_{N \to \infty}
\Bigl\| \dfrac{1}{|\Phi_N|} & \sum_{n \in \Phi_N}  T^n f_1  \otimes \dots \otimes T^n f_{k} \Bigr\|_{L^2(\nu)}
    \\
    &\leq
    \frac{1}{c}\sum_{j=0}^{c-1}\left(
    \limsup_{N\to\infty}\Bigl\| \dfrac{1}{|\Psi_N|}  \sum_{n \in \Psi_N}  T^{cn+j} f_1  \otimes \dots \otimes T^{cn+j} f_{k} \Bigr\|_{L^2(\nu)}\right)
\end{align*}
and so it suffices to show
\begin{equation}
\label{eq_thm_sigma_prop_V3}
\limsup_{N \to \infty}
\Bigl\| \dfrac{1}{|\Psi_N|} \sum_{n \in \Psi_N}  T^{cn+j} f_1  \otimes \dots \otimes T^{cn+j} f_{k}\Bigr\|_{L^2(\nu)}
\le
D_k \min \{ \|f_i\|_{U^k(X,\mu,T)} : 1\leq i\leq k \}
\end{equation}
for each fixed integer $j\leq c$.
For each $1 \le i \le k$, define $g_i \colon  X^{k} \to \C$ by
\[
g_i(x_1,\dots,x_{k}) = (T^jf_i)(x_i).
\]
Set $S=T\times T^2\times\cdots\times T^k$ and note that
\[
S^{c_in}g_i=\one\otimes\cdots\otimes \one\otimes T^{cn+j}f_i\otimes\one\otimes\cdots\otimes\one
\]
for all $n \in \N$.
Thus we can rewrite~\eqref{eq_thm_sigma_prop_V3} as
\[
\limsup_{N \to \infty}
\Bigl\| \dfrac{1}{|\Psi_N|} \sum_{n \in \Psi_N}  S^{c_1n} g_1 \cdots S^{c_kn} g_k \Bigr\|_{L^2(\nu)}
\le
D_k \min \{ \|f_i\|_{U^k(X,\mu,T)} : 1\leq i \leq k \}.
\]
By part~\ref{itm_uniformity_seminorms_properties_iii} of \cref{lem_uniformity_seminorms_properties} applied to the system $(X^k, \nu, S)$ and our choice of $D_k$, we have that
\[
\limsup_{N \to \infty}
\Bigl\| \dfrac{1}{|\Psi_N|} \sum_{n \in \Psi_N}  S^{c_1n} g_1 \cdots S^{c_kn} g_k \Bigr\|_{L^2(\nu)}
\le
D_k \min \{ \|g_i\|_{U^k(X^k,\nu,S)} : 1\leq i \leq k \}.
\]
For $i=1,\dots,k$, we have
\[
\|g_i\|_{U^k(X^k,\nu,S)}=\|T^jf_i\|_{U^k(X,\nu_i,T^i)}
\]
and so the conclusion  follows by applying part~\ref{itm_uniformity_seminorms_properties_iv} of \cref{lem_uniformity_seminorms_properties}.
\end{proof}

\begin{lemma}
\label{lemma_removingb}
Let $(Y,\tau)$ be a probability space, $\Phi\subset\N$ be a finite set, $b\colon \Phi\to\C$, and $(v_n)_{n\in\Phi}$ be a collection of vectors in $L^2(\tau)$.
Then
\[
\left\|\dfrac{1}{|\Phi|} \sum_{n \in \Phi} b(n)v_n\right\|_{L^2(\tau)}^2
    \leq \|b\|_\infty^2\cdot\left\|\dfrac{1}{|\Phi|} \sum_{n \in \Phi}v_n\otimes\overline{v_n}\right\|_{L^2(\tau\times\tau)}.
\]
\end{lemma}
\begin{proof}
We expand the square and obtain    \begin{equation}\label{eq_lemmathatseemstobevdcbutisnot}
\left\|\dfrac{1}{|\Phi|} \sum_{n \in \Phi}b(n)v_n\right\|_{L^2(\tau)}^2
        =     \dfrac{1}{|\Phi|^2} \sum_{n,m  \in \Phi}b(n)\overline{b(m)}\int_Yv_n\overline{v_m}\d\tau
        \leq
    \|b\|_\infty^2\cdot\dfrac{1}{|\Phi|^2} \sum_{n,m \in \Phi}\left|\int_Yv_n\overline{v_m}\d\tau\right|.
    \end{equation}
  Using the Cauchy-Schwarz inequality, we  estimate
    \begin{eqnarray*}
\left(\dfrac{1}{|\Phi|^2} \sum_{n,m \in \Phi}\left|\int_Yv_n\overline{v_m}\d\tau\right|\right)^2
    &\leq&
\dfrac{1}{|\Phi|^2} \sum_{n,m \in \Phi}
\left|\int_Yv_n\overline{v_m}\d\tau\right|^2 \\
\    &  = &
   \dfrac{1}{|\Phi|^2} \sum_{n,m \in \Phi}\int_{Y\times Y}\big(v_n\otimes\overline{v_n}\big)\cdot\big(\overline{v_m}\otimes v_m\big)\d(\tau\times\tau)
    \\&=&
    \int_{Y\times Y}\left|\dfrac{1}{|\Phi|} \sum_{n \in \Phi}v_n\otimes\overline{v_n}\right|^2\d(\tau\times\tau)
    =
    \left\|\dfrac{1}{|\Phi|} \sum_{n \in \Phi}v_n\otimes\overline{v_n}\right\|_{L^2(\tau\times\tau)}^2.
    \end{eqnarray*}
  Combining this  with~\eqref{eq_lemmathatseemstobevdcbutisnot}, we obtain the conclusion.
\end{proof}

\begin{proof}[Proof of \cref{thm_sigma_prop_iv}]
Let $\nu\in\M((X\times X)^{k-1})$ be the pushforward of $\tau\times\tau$ under the map
$$\varphi\colon (x_0,\dots,x_k,y_0,\dots,y_k)\mapsto(x_1,y_1,x_2,y_2,\dots,x_{k-1},y_{k-1}).$$
Note that $\nu$ is invariant under $(T\times T)\times (T^2\times T^2)\times\dots\times(T^{k-1}\times T^{k-1})$ and that each of the $k-1$ marginals $\nu_i\in\M(X\times X)$ of $\nu$ are $\nu_i=\tau_i\times\tau_i$ and hence satisfy $\nu_i\leq i^2(\mu\times\mu)$ for $i=1,\dots,k-1$.
Applying \cref{thm_sigma_prop_V} to the system $(X \times X, \mu \times \mu, T \times T)$ with $k-1$ instead of $k$ and $f_i \otimes \overline{f_i}$ instead of $f_i$, it follows that
\begin{equation*}
\begin{aligned}
\limsup_{N \to \infty}
\Bigl\| \dfrac{1}{|\Phi_N|} \sum_{n \in \Phi_N} (T\times T)^n (f_1\otimes \overline{f_1})  &{} \otimes \dots \otimes (T\times T)^n (f_{k-1}\otimes\overline{ f_{k-1}})\Bigr\|_{L^2(\nu)}
\\
\leq {}&
D_{k-1} \cdot \min \Big\{ \big\|f_i\otimes\overline{f_i}\big\|_{U^{k-1}(X\times X,\mu\times\mu,T\times T)} : 1 \le i \le k-1 \Big\}.
\end{aligned}
\end{equation*}
Part~\ref{itm_uniformity_seminorms_properties_i} of \cref{lem_uniformity_seminorms_properties} implies that
\[
\|f_i\otimes\overline{f_i}\|_{U^{k-1}(X\times X,\mu\times\mu,T\times T)}\leq \|f_i\|_{U^k(X,\mu\,T)}^2
\]
and so we deduce that
\begin{equation}
\label{eq_thm_sigma_prop_iv_7}
\begin{aligned}
\limsup_{N \to \infty}
\Bigl\| \dfrac{1}{|\Phi_N|} \sum_{n \in \Phi_N} (T\times T)^n (f_1\otimes \overline{f_1})  &{} \otimes \dots \otimes (T\times T)^n (f_{k-1}\otimes\overline{ f_{k-1}})\Bigr\|_{L^2(\nu)}
\\
\leq {}&
D_{k-1} \cdot \min \big\{\|f_i\|_{U^k(X,\mu\,T)}^2  : 1 \le i \le k-1 \big\}.
\end{aligned}
\end{equation}
Setting
\[
v_n=\one\otimes T^n f_1 \otimes \dots \otimes T^n f_{k-1}\otimes \one\in L^2(\tau),
\]
we note that
\[
v_n\otimes\overline{v_n} = \left( (T\times T)^n (f_1\otimes \overline{f_1}) \otimes \dots \otimes (T\times T)^n (f_{k-1}\otimes\overline{ f_{k-1}}) \right) \circ\varphi.
\]
Since $\nu=\varphi_*(\tau \times \tau) $, combining \cref{lemma_removingb} (applied to each $\Phi=\Phi_N$) with~\eqref{eq_thm_sigma_prop_iv_7}, it follows that
\begin{equation*}
\limsup_{N \to \infty}\left\| \dfrac{1}{|\Phi_N|} \sum_{n \in \Phi_N} b(n) \cdot  v_n\right\|_{L^2(\tau)}^2
\leq {}
\|b\|_\infty^2\cdot D_{k-1} \cdot \min \big\{ \|f_i\|_{U^k(X,\mu\,T)}^2 : 1 \le i \le k-1 \big\}.
\end{equation*}
Taking square roots on both sides, we obtain the desired conclusion with $C_k=\sqrt{D_{k-1}}$.
\end{proof}

\subsection{Proof of \cref{sigma_coordinate_invariant}}
\label{sec_sigma_coordinate_invariant}

For the proof of \cref{second_main_theorem_about_sigma} we need one more lemma.

\begin{lemma}
\label{lemma_correlationorthogonality2}
Let $(X,\mu,T)$ be an ergodic system, let $a\in\gen(\mu,\Phi)$ for some \Folner{} sequence $\Phi$, and denote by $(Z_{k-1}, m_{k-1},T)$ the $(k-1)$-step pronilfactor of $(X,\mu,T)$.
Assume $(Y,S)$ is a $(k-1)$-step pronilsystem.
Then for every $g\in C(X)$, $y\in Y$, and $F\in C(Y)$,
\begin{equation}
\label{eq_lemma_correlationorthogonality}
\limsup_{N\to\infty}\left|\frac1{|\Phi_N|}\sum_{n\in\Phi_N}g(T^na)F(S^ny)\right|
\leq
\big\|\E(g\mid Z_{k-1})\big\|_{L^1(m_{k-1})} \cdot\|F\|_{\infty}.
\end{equation}
\end{lemma}
We stress that~\eqref{eq_lemma_correlationorthogonality} requires the function $g$ to be continuous, as we could otherwise modify it on the orbit of $a$ without changing the right-hand side.
\begin{proof}
Pick an increasing sequence $(N_j)_{j\in\N}$ of natural numbers such that
\[
\limsup_{N\to\infty}  \left|\frac1{|\Phi_N|}\sum_{n\in\Phi_N} g(T^n a)F(S^ny)\right|
=
\lim_{j\to\infty} \left|\frac1{|\Phi_{N_j}|}\sum_{n\in\Phi_{N_j}} g(T^n a)F(S^ny)\right|.
\]
By refining $(N_j)_{j\in\N}$ if necessary, we can assume that $(a,y)$ is generic along $(\Phi_{N_j})_{j\in\N}$ for a measure $\rho$ on $X\times Y$ with respect to $T\times S$.

Since $(Y,S)$ is a pronilsystem, there is an invariant measure $\nu$ on $Y$ for which $y \in \gen(\nu,\Phi)$.
Replacing $Y$ with the support of $\nu$ we may assume that $(Y,\nu,S)$ is ergodic.
Since $a\in\gen(\mu,\Phi)$ we have that the first marginal of $\rho$ is $\mu$ and the second marginal is $\nu$.
Letting $\tilde{g}=\E\bigl(g \mid Z_{k-1})\circ\pi_{k-1}$, we obtain
\begin{align*}
\lim_{j\to\infty} &\left|\frac1{|\Phi_{N_j}|}\sum_{n\in\Phi_{N_j}} g(T^n a)F(S^ny)\right|\\
&~=\,
\left|\int_{X\times Y} g\otimes F \d\rho\right|
\leq
\left|\int_{X\times Y} (g-\tilde{g})\otimes F \d\rho\right|+\left|\int_{X\times Y} \tilde{g}\otimes F \d\rho\right|.
\end{align*}
Note that
\[
\bigl\|(g-\tilde{g})\otimes\one\bigr\|_{U^k(X\times Y,\rho,T\times S)}
=
\bigl\|g-\tilde{g}\bigr\|_{U^k(X,\mu,T)} =0.
\]
Thus $(g - \tilde{g}) \otimes \one$ is orthogonal to the maximal $(k-1)$-step pronilfactor of $(X \times Y, \zeta, T \times S)$ for almost every ergodic component $\zeta$ of $\rho$ by \cite[Proposition 18, Chapter 8]{HK-book}.
On the other hand, since $(Y,\nu,S)$ is an ergodic  $(k-1)$-step nilsystem it is a factor of $(X \times Y, \zeta, T \times S)$.
Thus $(Y,\nu,S)$ is a factor of the maximal $(k-1)$-step pronilfactor of $(X \times Y, \zeta, T \times S)$.
Therefore, as $\one \times F$ is measurable with respect to the latter factor,
the functions $(g-\tilde{g})\otimes\one$ and $\one\otimes F$ are orthogonal in $L^2(\zeta)$ for almost every ergodic component $\zeta$ of $\rho$, and hence
\[
\int_{X\times Y} (g-\tilde{g})\otimes F \d\rho=0.
\]
We are left with
\begin{align*}
\limsup_{N\to\infty}  \left|\frac1{|\Phi_N|}\sum_{n\in\Phi_N} g(T^n a)F(S^ny)\right|&\leq \left|\int_{X\times Y} \tilde{g}\otimes F \d\rho\right|
\\
&\leq
\|F\|_\infty\cdot \int_{X} \left|\tilde{g}\right| \d\mu =\|F\|_\infty\cdot \int_{Z_{k-1}} \left|\E(g\mid Z_{k-1})\right| \d m_{k-1},
\end{align*}
and the claim follows.
\end{proof}

\begin{proof}[Proof of \cref{second_main_theorem_about_sigma}]
By a standard approximation argument, it suffices to show that for all  $f_1,\dots,f_k\in C(X)$ with $\|f_i\|_\infty\leq1$,
\begin{equation}
\label{eqn_prf_2ndmt_s_1}
\lim_{N\to\infty}\biggl\|\frac1{|\Phi_N|}\sum_{n\in\Phi_N}\TT^n (f_1\otimes\ldots\otimes f_k\otimes\one) \,-\, \frac1{|\Phi_N|}\sum_{n\in\Phi_N} \TT^n(\one\otimes f_1\otimes\ldots\otimes f_k)\biggr\|_{L^2(\sigma)}=0.
\end{equation}
To ease the notation,
we write $g_i=\E(f_i\mid Z_{k-1})\circ\pi_{k-1}$ for $i=1,\ldots,k$.
Note that for every $n\in\N$
\begin{align*}
\TT^n (f_1\otimes\ldots\otimes f_k\otimes\one)
=
f_1(T^na)\cdot \TT^n(\one\otimes f_2\otimes \ldots\otimes f_k\otimes\one)
\end{align*}
holds $\sigma$-almost everywhere.
Using~\eqref{eq_OneLineStructureTheory}, we have that $\|f_i-g_i\|_{U^k(X,\mu,T)}=0$ for each $i$.
In view of \cref{lem_sigma_marginals}, $\sigma$ satisfies the assumptions of \cref{thm_sigma_prop_iv}, which now implies that
\[
\begin{split}
\lim_{N\to\infty}\biggl\|\frac1{|\Phi_N|}\sum_{n\in\Phi_N} & f_1(T^na)\cdot\TT^n(\one\otimes f_2\otimes \ldots\otimes f_k\otimes\one)
\\
&\,-\, \frac1{|\Phi_N|}\sum_{n\in\Phi_N} f_1(T^na)\cdot \TT^n(\one\otimes g_2\otimes \ldots\otimes g_k\otimes\one)\biggr\|_{L^2(\sigma)}=0.
\end{split}
\]
On the other hand, \cref{thm_sigma_prop_V} implies
\[
\lim_{N\to\infty}\biggl\|\frac1{|\Phi_N|}\sum_{n\in\Phi_N} \TT^n(\one\otimes f_1\otimes\ldots\otimes f_k)  \,-\, \frac1{|\Phi_N|}\sum_{n\in\Phi_N} \TT^n(\one\otimes g_1\otimes\ldots\otimes g_k)
 \biggr\|_{L^2(\sigma)}=0.
\]
This shows that~\eqref{eqn_prf_2ndmt_s_1} is equivalent to
\begin{equation}
\label{eqn_prf_2ndmt_s_2}
\begin{split}
\lim_{N\to\infty}\biggl\|\frac1{|\Phi_N|}\sum_{n\in\Phi_N} & f_1(T^na)\cdot \TT^n(\one\otimes g_2\otimes \ldots\otimes g_k\otimes\one)
\\
&\,-\, \frac1{|\Phi_N|}\sum_{n\in\Phi_N} \TT^n(\one\otimes g_1\otimes \ldots\otimes g_k)\biggr\|_{L^2(\sigma)}=0.
\end{split}
\end{equation}
Using the definition of $\sigma$, we can rewrite~\eqref{eqn_prf_2ndmt_s_2} as
\begin{equation}
\label{eqn_prf_2ndmt_s_3}
\begin{split}
\lim_{N\to\infty}\biggl\|\frac1{|\Phi_N|}\sum_{n\in\Phi_N} & f_1(T^na)\cdot \big(T^n \E(f_2\mid Z_{k-1})\otimes \ldots\otimes T^n \E(f_k\mid Z_{k-1})\otimes\one\big)
\\
&\,-\, \frac1{|\Phi_N|}\sum_{n\in\Phi_N} \big(T^n \E(f_1\mid Z_{k-1})\otimes \ldots\otimes T^n \E(f_k\mid Z_{k-1})\big)\biggr\|_{L^2(\xi)}=0,
\end{split}
\end{equation}
where $\xi$ is as in~\eqref{def:xi}.
Fix $1 > \epsilon>0$ and choose  $h_1,\dots,h_k\in C(Z_{k-1})$ such that
\begin{equation}
\label{eq_proof_thm_sigma_progressive1}
\|\E(f_i\mid Z_{k-1})-h_i\|_{L^2(m_{k-1})}<\epsilon.
\end{equation}
As the functions $f_1,\dots,f_k$ are uniformly bounded by $1$, we can assume that the functions $h_1,\dots,h_k$ are as well.

To prove~\eqref{eqn_prf_2ndmt_s_3}, we  establish the following three statements:
\begin{align}
\label{eqn_prf_2ndmt_s_5}
\begin{split}
\limsup_{N\to\infty}\biggl\| \frac1{|\Phi_N|}\sum_{n\in\Phi_N} & \big(T^n \E(f_1\mid Z_{k-1})\otimes \ldots\otimes T^n \E(f_k\mid Z_{k-1})\big)
\\
&\,-\, \frac1{|\Phi_N|}\sum_{n\in\Phi_N} \big(T^n h_1\otimes \ldots\otimes T^n h_k\big) \biggr\|_{L^2(\xi)}\leq k^2\epsilon,
\end{split}
\end{align}
\begin{align}
\label{eqn_prf_2ndmt_s_4}
\begin{split}
\limsup_{N\to\infty}\biggl\| \frac1{|\Phi_N|}\sum_{n\in\Phi_N} & f_1(T^na)\cdot \big(T^n \E(f_2\mid Z_{k-1})\otimes \ldots\otimes T^n \E(f_k\mid Z_{k-1})\otimes\one\big)
\\
&\,-\, \frac1{|\Phi_N|}\sum_{n\in\Phi_N} h_1(\pi_{k-1}(T^n a))\cdot \big(T^n h_2\otimes \ldots\otimes T^n h_k\otimes\one\big) \biggr\|_{L^2(\xi)}\leq k^2\epsilon+\sqrt{\epsilon},
\end{split}
\end{align}
\begin{align}
\label{eqn_prf_2ndmt_s_6}
\begin{split}
\lim_{N\to\infty}\biggl\|\frac1{|\Phi_N|}\sum_{n\in\Phi_N} h_1(\pi_{k-1}(T^n a))\cdot & (T^n h_2\otimes\ldots\otimes T^n h_k\otimes\one) \\
&\,-\, \frac1{|\Phi_N|}\sum_{n\in\Phi_N} ( T^n h_1\otimes\ldots\otimes T^n h_k)\biggr\|_{L^2(\xi)}=0.
\end{split}
\end{align}
Since $\epsilon$ can be taken arbitrarily small, \eqref{eqn_prf_2ndmt_s_3} follows by combining~\eqref{eqn_prf_2ndmt_s_5}, \eqref{eqn_prf_2ndmt_s_4},  and~\eqref{eqn_prf_2ndmt_s_6} and using the triangle inequality.

To prove~\eqref{eqn_prf_2ndmt_s_5} we first combine~\eqref{eq_proof_thm_sigma_progressive1} with the fact that $m_{k-1}$ is $T$-invariant and~\eqref{eqn_ioxi} to deduce that for every $n\in\N$,
\begin{equation*}
\|T^n\E(f_i\mid Z_{k-1})-T^nh_i\|_{L^2(\xi_i)}< \epsilon \sqrt{i},
\end{equation*}
where $\xi_i$ is the $i$-th marginal of $\xi$.
This, in turn, implies that
\begin{equation*}
\label{eq_proof_thm_sigma_progressive3}
\big\|\big(T^n\E(f_1\mid Z_{k-1})\otimes\ldots\otimes T^n\E(f_k\mid Z_{k-1})\big) -\big(T^nh_1\otimes\ldots\otimes T^nh_k\big)\big\|_{L^2(\xi)}\leq k^2\epsilon
\end{equation*}
by estimating coordinate-wise, and hence~\eqref{eqn_prf_2ndmt_s_5} follows by averaging.

To prove~\eqref{eqn_prf_2ndmt_s_4} we first note that, by the same argument, we have
\begin{equation}
\label{eq_proof_thm_sigma_progressive3.2}
\Big\|\big(T^n\E(f_2\mid Z_{k-1})\otimes\ldots\otimes T^n\E(f_k\mid Z_{k-1})\otimes\one\big) -\big(T^nh_2\otimes\ldots\otimes T^nh_k\otimes\one\big)\Big\|_{L^2(\xi)}\leq k^2\epsilon.
\end{equation}
Defining $g(x)=f_1(x)-h_1(\pi_{k-1}(x))$, we then have that
\[
f_1(T^na)-h_1(\pi_{k-1}(T^n a))=g(T^n a).
\]
Since $\pi_{k-1}$ is continuous, the function $g$ is continuous.
Multiplying~\eqref{eq_proof_thm_sigma_progressive3.2} by $f_1(T^na)$ and using the triangle inequality, we reduce~\eqref{eqn_prf_2ndmt_s_4} to
\begin{align}
\label{eqn_prf_2ndmt_s_7}
\begin{split}
\limsup_{N\to\infty}\biggl\| \frac1{|\Phi_N|}\sum_{n\in\Phi_N} & g(T^n a)\cdot \big(T^n h_2\otimes \cdots\otimes T^n h_k\otimes\one\big)
 \biggr\|_{L^2(\xi)}\leq \sqrt{\epsilon}.
\end{split}
\end{align}
To establish~\eqref{eqn_prf_2ndmt_s_7} we use \cref{lemma_correlationorthogonality2}.
Defining $Y=Z_{k-1}^{k}$, $S=T\times\cdots \times T$, and $F(x_1,\ldots,x_k)=h_2(x_1)\cdots h_k(x_{k-1})$, we then have
\begin{align*}
\limsup_{N\to\infty}\biggl\| \frac1{|\Phi_N|}\sum_{n\in\Phi_N} & g(T^n a)\cdot \big(T^n h_2\otimes \cdots\otimes T^n h_k\otimes\one\big)
 \biggr\|_{L^2(\xi)}^2
\\
&\leq \limsup_{N\to\infty}\biggl\| \frac1{|\Phi_N|}\sum_{n\in\Phi_N} g(T^n a)\cdot \big(T^n h_2\otimes \cdots\otimes T^n h_k\otimes\one\big)
 \biggr\|_{L^1(\xi)}
\\
&\leq \int_Y \limsup_{N\to\infty}\bigg| \frac1{|\Phi_N|}\sum_{n\in\Phi_N} g(T^n a)\cdot F(S^ny) \bigg|d\xi(y)
\\
&\leq \sup_{y\in Y}\limsup_{N\to\infty} \bigg| \frac1{|\Phi_N|}\sum_{n\in\Phi_N} g(T^n a)\cdot F(S^ny)\bigg|.
\end{align*}
Since $\pi_{k-1}$ is continuous, the function $g$ is continuous and therefore we can apply \cref{lemma_correlationorthogonality2} to conclude that
\begin{multline*}
\limsup_{N\to\infty} \left|\frac1{|\Phi_N|}\sum_{n\in\Phi_N} \big(f_1(T^na)-h_1(\pi_{k-1}(T^n a))\big)\cdot F(S^ny)
\right|
\\  \leq \big\|\E(g\mid Z_{k-1})\big\|_{L^1(m_{k-1})} \cdot\|F\|_{\infty}.
\leq \big\|\E(f_1\mid Z_{k-1})-h_1\big\|_{L^2(m_{k-1})}\cdot\|F\|_{\infty}
\leq \epsilon,
\end{multline*}
and the estimate in~\eqref{eqn_prf_2ndmt_s_7} follows.

To finish the proof, we are left with verifying~\eqref{eqn_prf_2ndmt_s_6}. For this we use the definition~\eqref{def:xi} of $\xi$ and the fact that orbit closures in pronilsystems are uniquely ergodic pronilsystems.
We have
\begin{align*}
\lim_{N\to\infty} & \biggl\|\frac1{|\Phi_N|}\sum_{n\in\Phi_N} h_1\bigl(\pi_{k-1}(T^n a)\bigr)\cdot (T^n h_2\otimes\dots\otimes T^n h_k\otimes\one)
\,-\, \frac1{|\Phi_N|}\sum_{n\in\Phi_N} ( T^n h_1\otimes\dots\otimes T^n h_k)\biggr\|_{L^2(\xi)}^2
\\
&=\lim_{N\to\infty}\int_{Z_{k-1}^k} \bigg| \frac1{|\Phi_N|}\sum_{n\in\Phi_N} h_1\bigl(\pi_{k-1}(T^n a)\bigr)\cdot (T^n h_2\otimes\dots\otimes T^n h_k\otimes\one)
\\
&\qquad\qquad\qquad\qquad\qquad\qquad\qquad- \frac1{|\Phi_N|}\sum_{n\in\Phi_N}( T^n h_1\otimes\dots\otimes T^n h_k) \bigg|^2\d\xi
\\
\\
&=\lim_{N\to\infty}\lim_{M\to\infty}\frac1{|\Phi_M|}\sum_{m\in\Phi_M} \bigg| \frac1{|\Phi_N|}\sum_{n\in\Phi_N} h_1\bigl(T^n(\pi_{k-1} a)\bigr)h_2\bigl(T^{n+m}(\pi_{k-1} a)\bigr) \cdots h_k\bigl(T^{n+(k-1)m}(\pi_{k-1} a)\bigr)
\\
&\qquad\qquad\qquad\qquad\qquad\qquad- \frac1{|\Phi_N|}\sum_{n\in\Phi_N} h_1\bigl(T^{n+m}(\pi_{k-1} a)\bigr) \cdots h_k\bigl(T^{n+km}(\pi_{k-1} a)\bigr) \bigg|^2.
\end{align*}
Since $Z_{k-1}$ is a pronilsystem, arguing as at the end of the proof of \cref{sigma_diagonal_average_furstenberg} we can interchange the order of the limits, and this last expression becomes
\begin{align*}
\lim_{M\to\infty}\lim_{N\to\infty} & \frac1{|\Phi_M|}\sum_{m\in\Phi_M} \bigg| \frac1{|\Phi_N|}\sum_{n\in\Phi_N} h_1\bigl(T^n(\pi_{k-1} a)\bigr) h_2\bigl(T^{n+m}(\pi_{k-1} a)\bigr) \dots h_k\bigl(T^{n+(k-1)m}(\pi_{k-1} a)\bigr)
\\
&\qquad- \frac1{|\Phi_N|}\sum_{n\in\Phi_N} h_1\bigl(T^{n+m}(\pi_{k-1} a)\bigr)
h_1\bigl(T^{n+2m}(\pi_{k-1} a)\bigr)
\cdots h_k\bigl(T^{n+km}(\pi_{k-1} a)\bigr) \bigg|^2.
\end{align*}
Making the change of variables $n\mapsto n-m$ in the second term shows that the two limits are identical, concluding the proof.
\end{proof}

\bigskip
\footnotesize
\noindent
Bryna Kra \\
\textsc{Northwestern University}\par\nopagebreak
\noindent
\href{mailto:kra@math.northwestern.edu}
{\texttt{kra@math.northwestern.edu}}

\bigskip
\footnotesize
\noindent
Joel Moreira\\
\textsc{University of Warwick} \par\nopagebreak
\noindent
\href{mailto:joel.moreira@warwick.ac.uk}
{\texttt{joel.moreira@warwick.ac.uk}}

\bigskip
\footnotesize
\noindent
Florian K.\ Richter\\
\textsc{{\'E}cole Polytechnique F{\'e}d{\'e}rale de Lausanne} (EPFL)\par\nopagebreak
\noindent
\href{mailto:f.richter@epfl.ch}
{\texttt{f.richter@epfl.ch}}

\bigskip
\footnotesize
\noindent
Donald Robertson\\
\textsc{University of Manchester} \par\nopagebreak
\noindent
\href{mailto:donald.robertson@manchester.ac.uk}
{\texttt{donald.robertson@manchester.ac.uk}}

\let\thefootnote\relax\footnotetext{For the purpose of open access, the authors have applied  a Creative Commons Attribution (CC BY) license to any Author Accepted Manuscript version arising from this submission.}

\end{document}